\documentclass[12pt, a4paper, twoside]{article}

\usepackage[ngerman,english]{babel}	

\usepackage{titling}
\usepackage{authblk} 

\usepackage{graphicx}
\usepackage{amsmath} 
\usepackage{amssymb} 
\usepackage{amsfonts}
\usepackage{amsthm}
\usepackage{enumerate}
\usepackage{makeidx} 
\usepackage{verbatim}
\usepackage{epsfig}
\usepackage{hyperref} 
\usepackage{dsfont}
\usepackage{wasysym} 
\usepackage{caption}	
\usepackage{float}
\usepackage{stackrel}	
\usepackage{mathtools} 
\usepackage{ltablex}	

\usepackage{tikz}
\usepackage{pgfplots}

\usepackage{hyperref}
\hypersetup{colorlinks=true,
						linkcolor=black,
						citecolor=black,
						urlcolor=blue}

\makeindex

\makeindex      
     
\newlength{\fixboxwidth}     
\newcommand{\R}{{\mathbb R}}

\newcommand{\N}{{\mathbb N}}
\newcommand{\Z}{{\mathbb Z}}

\renewcommand{\P}{\mathbb P}
\newcommand\expect{\operatorname{\mathbb{E}}}

\DeclareMathOperator\Cov{Cov}

\DeclareMathOperator\Uniform{unif}
\DeclareMathOperator\Normal{\mathcal{N}}

\newcommand{\dint}{\,\mathup{d}}
\newcommand{\diff}{\mathup{d}} 

\newcommand{\Torus}{{\mathbb T}}

\newcommand{\Hilbert}{{\mathcal H}}
\newcommand\Korobov{\textup{Kor}}

\DeclareMathOperator\linspan{span}

\DeclareMathOperator\Vol{Vol}

\DeclareMathOperator*\esssup{ess\,sup}
\DeclareMathOperator\id{id}

\DeclareMathOperator\modulo{mod}

\newcommand\ran{\textup{ran}}
\newcommand\deter{\textup{det}}

\DeclareMathOperator\diam{diam}
\DeclareMathOperator\App{APP}
\DeclareMathOperator\Int{INT}

\newcommand\eps{\varepsilon}
\newcommand\euler{\mathup{e}}
\newcommand\imag{\mathup{i}}
\newcommand\embed{\hookrightarrow}

\newcommand{\zeros}{\boldsymbol{0}}

\newcommand\veck{\mathbf{k}}

\newcommand\vecx{\mathbf{x}}
\newcommand\vecy{\mathbf{y}}

\newcommand\vecX{\mathbf{X}}

\newcommand\veclambda{\boldsymbol{\lambda}}

\DeclareMathAlphabet{\mathup}{OT1}{\familydefault}{m}{n}

\usepackage{microtype} 

\usepackage{bookmark}
\makeatletter 
  \providecommand*{\toclevel@author}{999}
  \providecommand*{\toclevel@title}{0}
\makeatother

\theoremstyle{plain}       
     
\newtheorem{theorem}{Theorem}[section]      
      
\newtheorem{lemma}[theorem]{Lemma}     
\newtheorem{proposition}[theorem]{Proposition}      
     
\theoremstyle{definition}

\newtheorem*{notation}{Notation}
\newtheorem{example}[theorem]{Example}     
\newtheorem{remark}[theorem]{Remark}

\numberwithin{equation}{section}   

\numberwithin{theorem}{section}

\newcommand{\thmref}[1]{\hyperref[#1]{Theorem~\ref*{#1}}}
\newcommand{\lemref}[1]{\hyperref[#1]{Lemma~\ref*{#1}}}
\newcommand{\propref}[1]{\hyperref[#1]{Proposition~\ref*{#1}}}
\newcommand{\corref}[1]{\hyperref[#1]{Corollary~\ref*{#1}}}
\newcommand{\remref}[1]{\hyperref[#1]{Remark~\ref*{#1}}}
\newcommand{\chapref}[1]{\hyperref[#1]{Chapter~\ref*{#1}}}
\newcommand{\secref}[1]{\hyperref[#1]{Section~\ref*{#1}}}
\newcommand{\defref}[1]{\hyperref[#1]{Definition~\ref*{#1}}}
\newcommand{\exref}[1]{\hyperref[#1]{Example~\ref*{#1}}}

\definecolor{darkgrey}{rgb}{0.4, 0.4, 0.4}
\definecolor{lightgrey}{rgb}{0.75,0.75,0.75}

\begin{document}

\title{Breaking the Curse for Uniform Approximation
	in Hilbert Spaces
	via Monte Carlo Methods}
\author{Robert J. Kunsch\thanks{E-mail: robert.kunsch@uni-osnabrueck.de}}
\affil{Universit\"at Osnabr\"uck, Institut f\"ur Mathematik,
	Albrechtstr.~28a, 49076 Osnabr\"uck, Germany}
\date{\today}
\maketitle

\begin{abstract}
	We study the $L_{\infty}$-approximation of $d$-variate functions
	from Hilbert spaces via linear functionals as information.
	It is a common phenomenon in tractability studies that unweighted problems
	(with each dimension being equally important)
	suffer from the curse of dimensionality in the deterministic setting,
	that is, the number $n(\eps,d)$ of information needed in order to solve
	a problem to within a given accuracy~$\eps$ grows exponentially in~$d$.
	We show that for certain approximation problems
	in periodic tensor product spaces,
	in particular Korobov spaces with smoothness~$r > 1/2$,
	switching to the randomized setting can break the curse of dimensionality,
	now having polynomial tractability,
	namely $n(\eps,d) \preceq \eps^{-2} \, d \, (1 + \log d)$.
	Similar benefits of Monte Carlo methods in terms of tractability
	have only been known for integration problems so far.
\end{abstract}


\section{Introduction}

Concerning the problem of computing
the integral $\Int(f) := \int_{[0,1]^d} f(\vecx) \dint\vecx$
of a function~$f: [0,1]^d \to \R$
based on information from $n$~function evaluations,
it is known that for many classes of input functions
deterministic methods suffer from the \emph{curse of dimensionality},
that is, the number~$n^{\deter}(\eps,d)$ of function values needed
in order to guarantee some given accuracy~\mbox{$\eps > 0$}
grows exponentially in~$d$.
One example of such input classes where the curse holds,
are $C^r$-functions
with bounded partial derivatives, see~\cite{HNUW17}.
Another type of problems can be formulated with functions
that are bounded in the norm of certain unweighted
tensor product Hilbert spaces,
see~\cite[Chapter~16]{NW10}.
Usually, the standard Monte Carlo method,
\begin{equation} \label{eq:stdMC}
	M_n(f) := \frac{1}{n} \sum_{i=1}^n f(\vecX_i) \,,
	\qquad \vecX_i \stackrel{\textup{iid}}{\sim{}} \Uniform([0,1]^d) \,,
\end{equation}
can be used to bound the complexity in the randomized setting,
\mbox{$n^{\ran}(\eps,d) \leq \lceil \eps^{-2} \rceil$},
if functions under consideration are bounded in the~$L_2$-norm,
$\|f\|_2 \leq 1$.
Offering an upper bound which is independent from~$d$
and polynomial in~$\eps^{-1}$,
these problems are \emph{strongly polynomially tractable}
when using Monte Carlo.
It has been unknown so far
whether there exist $d$-variate function approximation problems
where Monte Carlo methods may help in a similar way.
Function approximation problems seem to be more difficult
than integration problems,
because instead of just a single real number (namely the value of an integral)
we aim to find a representation for an approximating function.
The question whether Monte Carlo does help, however,
is independent from that.
There are examples of function approximation problems
where Monte Carlo methods still suffer from the curse of dimensionality,
namely uniform approximation of certain classes of $C^{\infty}$-functions,
see~\cite{Ku16}, or \cite[Section~2.4.2]{Ku17} for an extended discussion.
In contrast, the aim of the present paper is to present an example of
Function approximation problems where the curse of dimensionality
does hold for deterministic methods,
but Monte Carlo methods achieve \emph{polynomial tractability}
with~\mbox{$n^{\ran}(\eps,d) \preceq \eps^{-2} \, d \, (1 + \log d)$}.
Somewhere in between integration and approximation
is the paper of Heinrich and Milla~\cite{HeM11}
on indefinite integration of $L_p$-functions,
where the output is Banach space valued
and only randomized methods can solve the problem,
and actually provide polynomial tractability.

In this paper we study the uniform approximation in
Hilbert spaces~$\Hilbert^d$ of functions on a $d$-dimensional domain,
$\App: \Hilbert^d \embed L_{\infty}$, based on information from $n$ linear functionals.
In particular we consider
periodic functions defined on the $d$-torus~\mbox{$\Torus^d = [0,1)^d$}.
This type of embeddings
is a special case of a linear problem~$S: F \rightarrow G$
where~$F$ and $G$ are Banach spaces.
(In this paper, for simplicity, we restrict to Banach spaces over the reals.)
A general deterministic approximation method for such a problem
is a mapping
\mbox{$A_n : F \stackrel{N}{\to} \R^n \stackrel{\phi}{\to} G$}
with an information mapping~\mbox{$N(f) = (L_1(f),\ldots,L_n(f))$},
where the $L_i$ are continuous linear functionals acting on~$F$.
(One could think about choosing the information functionals adaptively,
but this will not be necessary in the context of this paper.)
The error of a deterministic method is defined by the worst case,
\begin{equation*}
	e(A_n,S)
		:= \sup_{\|f\|_F \leq 1} \|Sf - A_n(f)\|_G \,.
\end{equation*}
Randomized methods are families~\mbox{$(A_n^{\omega})_{\omega}$} of such mappings,
indexed by a random element~\mbox{$\omega \in \Omega$}
from a suitable probability space~\mbox{$(\Omega,\Sigma,\P)$},
usually we impose additional measurability assumptions.
The error for a single input is now averaged over~$\omega$, we define
\begin{equation*}
	e((A_n^{\omega})_{\omega},S)
		:= \sup_{\|f\|_F \leq 1} \expect \|Sf - A_n^{\omega}(f)\|_G \,.
\end{equation*}
For both settings we define the \emph{$n$-th minimal error},
that is the error achievable
by optimal algorithms that use $n$~pieces of information,
\begin{equation*}
	e^{\deter}(n,S) = \inf_{A_n} e(A_n,S) \,,
	\qquad \text{and} \qquad
	e^{\ran}(n,S) := \inf_{(A_n^{\omega})_{\omega}} e((A_n^{\omega})_{\omega},S) \,.
\end{equation*}
The \emph{initial error} $e(0,S) = \|S\|_{F \to G}$ coincides for both settings.
In tractability studies we mainly focus on the inverse notion
\emph{$\eps$-complexity}.
We define the
\emph{(worst case) deterministic complexity}~\mbox{$n^{\deter}(\eps,S)$},
that is the minimal~$n \in \N_0$
for which we can find a deterministic algorithm~$A_n$
that guarantees an error smaller than or equal~$\eps > 0$,
analogously we have the \emph{Monte Carlo complexity}~\mbox{$n^{\ran}(\eps,S)$}
using randomized algorithms.
For more details on these and related notions
from \emph{information-based complexity} (IBC),
we refer to the books~\cite{NW08,TWW88}.

We will see that under quite natural assumptions
on a $d$-dependent family of unweighted tensor product spaces~$\Hilbert^d$,
the deterministic complexity
\mbox{$n^{\deter}(\eps,\Hilbert^d \embed L_{\infty})$}
of uniform approximation grows exponentially in~$d$ for some fixed~$\eps > 0$.
In other words,
the \emph{curse of dimensionality} holds in the deterministic setting,
see \thmref{thm:curseperiodic}.
In contrast,
the Monte Carlo complexity~$n^{\ran}(\eps,\Hilbert^d \embed L_{\infty})$
will depend only polynomially on~$\eps^{-1}$ and $d$,
which proves \emph{polynomial tractability} in the randomized setting.
Hence, Monte Carlo breaks the curse,
see \propref{prop:MCUBperiodic}.
This holds in particular for so-called Korobov spaces with smoothness~$r > 1/2$,
see \thmref{thm:Korobov}.

In the deterministic setting, linear methods
\begin{equation*}
	A_n(f) := \sum_{i=1}^n L_i(f) \, g_i \,,
	\qquad L_i \in (\Hilbert^d)^*, \, g_i \in L_{\infty} \,,
\end{equation*}
are known to be optimal, namely for two reasons,
where each in its own right would suffice already:
\begin{itemize}
	\item the input set is the unit ball of a Hilbert space,
	\item the error is measured with respect to the $L_{\infty}$-norm,
\end{itemize}
see~\cite[Theorem~4.5 and~4.8]{NW08}.
In detail,
\mbox{$(g_i)_{i=1}^n$} will be an orthonormal system
with corresponding functionals
\mbox{$L_i(f) = \langle g_i,f\rangle_{\Hilbert^d}$},
so $A_n$ is an orthogonal projection.
For periodic functions, the information we use
consists of Fourier coefficients,
which is usually regarded as a natural type of information.

In the randomized setting we will consider linear methods as well.
We propose methods where we take the average over $n$ independent ``trials''
to fit the function~$f$,
\begin{equation*}
	A_n^{\omega}(f)
		:= \frac{1}{n} \sum_{i=1}^n L_i^{\omega}(f) \, g_i^{\omega} \,,
\end{equation*}
where $g_i^{\omega} \in L_{\infty}$ are i.i.d.\ realizations
of a Gaussian field,
and the $L_i^{\omega}$ are corresponding ``Gaussian'' functionals,
which -- in a way -- indicate how well the~$g_i^{\omega}$ fit
to the original function~$f$.
This method resembles the simple structure
of standard Monte Carlo integration~\eqref{eq:stdMC},
however, the type of linear information we use here is quite exotic.
For periodic functions one can put it like this:
All Fourier coefficients simultaneously have a random influence
on each of the functionals.

Details of this randomized approximation method will be presented in
Section~\ref{sec:plainMC}.
The original idea goes back to Math\'e 1991~\cite{Ma91}
who considered finite-dimensional sequence space embeddings
\mbox{$\ell_2^m \embed \ell_q^m$}, \mbox{$q > 2$},
as building blocks in order to study the speed of convergence
for the randomized $L_q$-approximation of functions from
one-dimensional Sobolev spaces via optimal linear methods.
Heinrich 1992~\cite{He92} continued this work,
providing nonlinear Monte Carlo approximation methods
for certain parameter settings of $d$-variate isotropic spaces.
Similar results for periodic spaces of dominating mixed smoothness
by Fang and Duan~\cite{FD07} 
still rely on Math\'e's sequence space result,
but for some cases that have been left open there,
modified techniques needed to be applied, see~\cite{BKN17}.
To summarize, for $L_q$-approximation of functions from a Sobolev space
with integrability index~$p$,
randomization turns out to speed up the convergence if~$\max\{2,p\} < q$.
The biggest gain was made for~$p=2$ and $q = \infty$,
i.e.\ uniform approximation of Hilbert space functions,
where we can gain a speed-up of $n^{-1/2}$
for the decay of the error $e^{\ran}(n,\Hilbert^d \embed L_{\infty})$
of optimal Monte Carlo algorithms,
compared to the best possible deterministic error
\mbox{$e^{\deter}(n,\Hilbert^d \embed L_{\infty})$},
aside from logarithmic terms.
This reminds us of classical integration problems
where Monte Carlo would usually give an additional rate~$n^{-1/2}$
compared to deterministic cubature rules,
see Bakhvalov 1959~\cite{Bakh59},
also Novak~\cite[Sections~1.3.8/9 and 2.2.9]{No88}.

The problem of uniform approximation of functions from Hilbert spaces
forms a benchmark for the power of randomisation in terms of convergence rates,
and thereby motivate the study of the $d$-dependent complexity
in this particular setting.

The paper is organized as follows.
In \secref{sec:Basics} we collect general tools for understanding $L_{\infty}$-approximation
in general Hilbert spaces.
Namely,
we present the algorithmic background of the upper Monte Carlo bound we use,
we futher recall some theory on reproducing kernel Hilbert spaces and Gaussian fields,
and also cite a technique for lower bounds in the deterministic setting.
In \secref{sec:HilbertPeriodic} we apply these tools to the particular situation
of tensor product spaces of periodic functions, especially Korobov spaces.

\begin{notation}
	For a real number~$a$ we put $a_+ := \max\{a,0\}$.
	For functions~$f,g$ on a common domain,
	the notion~$f \preceq g$ means that there exists a constant~$C > 0$
	such that~$f \leq C g$.
	The symbol~\mbox{$f \asymp g$} will be used as an abbreviation for
	\mbox{$f \preceq g \preceq f$},
	the notion $f \prec g$ means $f \preceq g$ but $g \not\preceq f$.
	Note that the implicit constants in this paper
	are typically independent from the dimension~$d$
	and from the error threshold $\eps$,
	but may depend on other parameters such as the smoothness~$r$.
\end{notation}

\section{General Tools}
\label{sec:Basics}

\subsection{A plain Monte Carlo upper bound}
\label{sec:plainMC}

The following result originates from Math{\'e}~\cite{Ma91}
and is a key component for the Monte Carlo approximation
of Hilbert space functions via linear information.
Here we keep it a little more general than in the original paper,
where the output space was a sequence space~$\ell_q^m$ with~\mbox{$q>2$}.
See also~\cite{BKN17} for a proof of the general setting,
modifications for the complex setting, and further discussions.

\begin{proposition}
\label{prop:Ma91_l2G}
	Let \mbox{$S:\ell_2^m \rightarrow G$} be a linear operator between normed spaces
	and consider the unit ball~\mbox{$B_2^m \subset \ell_2^m$} as the input set.
	Let the information mapping \mbox{$N^{\omega} = N$}
	be a random \mbox{$(n \times m$)}-Matrix with entries
	\mbox{$N_{ij} = n^{-1/2} \, X_{ij}$},
	where the $X_{ij}$ are
	independent standard Gaussian random variables.
	Then \mbox{$A_n^{\omega} := S N^{\top} N$}
	defines a linear rank-$n$ Monte Carlo method
	(\mbox{$\phi^{\omega} = S N^{\top}$}),
	and its error is bounded from above by
	\begin{equation*}
		e(A_n,S:\ell_2^m \rightarrow G)
			\leq \frac{2 \expect \|S \vecX\|_G}{\sqrt{n}}
	\end{equation*}
	where $\vecX$ is a standard Gaussian vector in $\R^m$.
\end{proposition}

\begin{example}[Breaking the curse - a sequence space example]
	As a toy example of $d$-variate functions, consider
	functions that are defined on a Boolean domain, \mbox{$f :\{0,1\}^d \to \R$}.
	On the Boolean domain, $L_p$-spaces with respect to the counting measure
	translate into classical finite-dimensional sequence spaces~$\ell_p^{2^d}$,
	hence, the problem
	\begin{equation*}
		\ell_2^{2^d} \embed \ell_{\infty}^{2^d}
	\end{equation*}
	corresponds to ``uniform approximation in Hilbert spaces''.
	From Smolyak~\cite{Smo65} we extract
	\begin{equation*}
		n^{\deter}(\eps,\ell_2^{2^d} \embed \ell_{\infty}^{2^d})
			\geq (1-\eps^2) \, 2^d \,,
	\end{equation*}
	which implies the curse of dimensionality in the deterministic setting.
	The fundamental approximation method due to Math\'e~1991~\cite{Ma91}
	yields
	\begin{equation*}
		n^{\ran}(\eps,\ell_2^{2^d} \embed \ell_{\infty}^{2^d})
			\preceq \eps^{-2} \, d \,,
	\end{equation*}
	which proves polynomial tractability.
	(This result uses~$\expect \|\vecX\|_{\infty} \asymp \sqrt{1 + \log m}$
	for standard Gaussian vectors $\vecX \in \R^m$.)
	
	One could also study the general case
	\begin{equation*}
		\ell_p^{2^d} \embed \ell_q^{2^d} \,, \qquad 1 \leq p,q \leq \infty \,.
	\end{equation*}
	Much (but not everything) is known about deterministic approximation
	for this general problem, see~\cite{FR13CS,Ka81,Pie74,Smo65}.
	From Math\'e~\cite{Ma91} and Heinrich~\cite{He92}
	we obain linear and non-linear Monte Carlo methods,
	some parameter settings that have been left open there
	are easily filled via relations between different~$\ell_p$-norms.
	The lower Monte Carlo bounds from Heinrich~\cite{He92}
	will not give proper $d$-dependencies in most cases, unfortunately.
	See \cite[Section~3.2.2]{Ku17} for a detailed listing of all cases
	of sequence space embeddings,
	and discussions about lower bounds for Monte Carlo.\\
	The case~$p \geq q$ is fully understood, there is no way
	to reduce the initial error if we use $\prec 2^d$ pieces of information,
	even in the Monte Carlo setting.\\
	For $p \leq q$ the initial error is normalized,
	\mbox{$e(0,\ell_p^{2^d} \embed \ell_q^{2^d}) = 1$}.
	If $p = 1$ and $2 \leq q$, by methods related to compressed sensing
	we receive
	\begin{equation*}
		n^{\deter}(\eps,\ell_1^{2^d} \embed \ell_q^{2^d}) \preceq \eps^{-2} \, d \,,
	\end{equation*}
	so we already have polynomial tractability in the deterministic setting.
	For $1 < p$ and $\max\{p,2\} < q$, however,
	we can show the curse of dimensionality in the deterministic setting.
	It is only for~$1 < p \leq 2$ and $q = \infty$, though, that we know
	randomized approximation methods which break the curse.
	In detail, Math\'e's linear Monte Carlo method is already sufficient,
	even for~$1 < p < 2$, where nonlinear methods provide better results.
	
	This example is one more hint why it is promising
	to consider uniform approximation in Hilbert spaces
	in search for function approximation problems
	where Monte Carlo methods can break the curse.
\end{example}

We now put the fundamental Monte Carlo method from \propref{prop:Ma91_l2G}
to an extreme
and obtain a function approximation analogue
to standard Monte Carlo integration~\eqref{eq:stdMC}.
Here, we restrict the input set to functions from Hilbert spaces.
\lemref{lem:stdMCapp} stated below is still quite general,
its specification to $L_{\infty}$-approximation of functions
is the starting point for the study of Gaussian random fields
and their expected maximum, see \secref{sec:E|Psi|_sup}.

\begin{lemma} \label{lem:stdMCapp}
	Consider a linear problem with a compact solution operator
	\begin{equation*}
		S: \Hilbert \rightarrow G
	\end{equation*}
	from a separable Hilbert space~$\Hilbert$ into a Banach space~$G$.
	Assume that we have an orthonormal basis~\mbox{$(\psi_j)_{j \in \N}$}
	for~$\Hilbert$ such that
	the sum~\mbox{$\sum_{j=1}^{\infty} X_j \, S(\psi_j)$},
	with independent standard Gaussian random variables~\mbox{$X_j$},
	converges almost surely in~$G$.
	Then for~\mbox{$n \in \N$} we have
	\begin{equation*}
		e^{\ran}(n,S)
			\leq \frac{2 \expect \bigl\|\sum_{j=1}^{\infty}
														X_j \, S(\psi_j)
														\bigr\|_G
								}{\sqrt{n}} \,,
	\end{equation*}
	or equivalently, for~\mbox{$\eps > 0$},
	\begin{equation*}
		n^{\ran}(\eps,S)
			\leq \left\lceil
							4\, \left(\frac{\expect \bigl\|\sum_{j=1}^{\infty}
																	X_j \, S(\psi_j)
												\bigr\|_G
										}{\eps}
								\right)^2
					\right\rceil \,.
	\end{equation*}
\end{lemma}
\begin{proof}
	For~\mbox{$m \in \N$} we define the linear Monte Carlo method
	\mbox{$A_{n,m} = (A_{n,m}^{\omega})_{\omega}$}
	which, for an input \mbox{$f \in \Hilbert$},
	returns the output
	\begin{equation*}
		g = A_{n,m}^{\omega}(f)
			:= \frac{1}{n} \, \sum_{i=1}^n
						L_{i,m}^{\omega}(f) \, g_{i,m}^{\omega} \,,
	\end{equation*}
	based on the information
	\begin{equation*}
		y_i = L_{i,m}^{\omega}(f)
			=  \sum_{j=1}^m X_{ij} \, \langle \psi_j, f \rangle_{\Hilbert} \,,
	\end{equation*}
	and with elements from the output space
	\begin{equation*}
		g_{i,m}^{\omega} := \sum_{j=1}^m X_{ij} \, S(\psi_j) \,.
	\end{equation*}
	Here, the~$X_{ij}$ are independent standard Gaussian random variables.
	This algorithm is actually the fundamental Monte Carlo method
	from \propref{prop:Ma91_l2G}
	when restricting~$S$
	to	the subspace~\mbox{$\Hilbert_m := \linspan\{\psi_1,\ldots,\psi_m\}$}
	which can be identified with~$\ell_2^m$.
	Let~$P_m$ denote the orthogonal projection onto~$\Hilbert_m$.
	Compactness of~$S$ provides
	\begin{equation*}
		\|S \, (\id_{\Hilbert} - P_m) \|_{\Hilbert \rightarrow G}
			\xrightarrow[m \rightarrow \infty]{} 0 \,.
	\end{equation*}
	Then, for inputs~$f$ with \mbox{$\|f\|_{\Hilbert} \leq 1$},
	employing \propref{prop:Ma91_l2G} we have
	\begin{align*}
		e((A_{n,m}^{\omega})_{\omega},f)
			&\leq \|S \, (\id_{\Hilbert} - P_m) \, f \|_G
						+ e((A_{n,m}^{\omega})_{\omega}, P_m f) \\
			&\leq \|S(\id_{\Hilbert} - P_m) \|_{\Hilbert \rightarrow G}
				+ \frac{2}{\sqrt{n}}
						\, \expect\biggl\|\sum_{j=1}^m
																X_j \, S(\psi_j)
											\biggr\|_G\\
			&\xrightarrow[m \rightarrow \infty]{}
				\frac{2}{\sqrt{n}}
					\, \expect\biggl\|\sum_{j=1}^{\infty}
																X_j \, S(\psi_j)
										\biggr\|_G\,.
	\end{align*}
	Note that~\mbox{$\expect\left\|\sum_{j=1}^m X_j \, S(\psi_j) \right\|_G$}
	is monotonically increasing in~$m$,
	since the~$X_j$ are independent and centrally symmetric.
	In detail,
	\begin{align*}
		\expect\left\|\sum_{j=1}^m X_j \, S(\psi_j) \right\|_G
			&\stackrel{\text{$\Delta$-ineq.}}{\leq}
				\frac{1}{2} \sum_{\sigma = \pm 1}
					\left\|\sum_{j=1}^m X_j \, S(\psi_j) + \sigma X_{m+1} S(\psi_{m+1})
					\right\|_G \\
			&\stackrel{\phantom{\text{$\Delta$-ineq.}}}{=}
				\expect\left\|\sum_{j=1}^{m+1} X_j \, S(\psi_j) \right\|_G \,.
	\end{align*}
\end{proof}

\begin{remark}[Stochastically bounded information and algorithms]
	\label{rem:Lstoch}
	With the above lemma it seems natural to consider
	the idealized method~\mbox{$A_n = A_{n,\infty}$},
	\begin{equation*}
		A_n(f)
			:= \frac{1}{n} \, \sum_{i=1}^n L_i^{\omega}(f) \, g_i^{\omega}  \,,
	\end{equation*}
	with information
	\begin{equation*}
		y_i = L_i^{\omega}(f)
			:= \sum_{j=1}^{\infty} X_{ij} \, \langle \psi_j,f \rangle_{\Hilbert} \,,
	\end{equation*}
	where~$X_{ij}$ are i.i.d. standard Gaussian,
	and elements from the output space 
	\begin{equation*}
		g_i^{\omega} := \sum_{j=1}^{\infty} X_{ij} \, S(\psi_j) \,.
	\end{equation*}
	Observe the similarities
	with standard Monte Carlo integration~\eqref{eq:stdMC}.
	Observe also the important difference that the present approximation method
	depends on the particular norm of the input space,
	whereas standard Monte Carlo integration is defined
	independently from the input set.
	
	Note that, almost surely, $L_i^{\omega}$~is an unbounded functional.
	To see this,
	for fixed~$\omega$, consider the sequence~\mbox{$(f_{ik})_{k=1}^{\infty}$}
	of normalized Hilbert space elements
	\begin{equation*}
		f_{ik} := \frac{1}{\sqrt{\sum_{j=1}^k X_{ij}^2}}
						\, \sum_{j=1}^k X_{ij} \, \psi_j
			\in \Hilbert \,,
	\end{equation*}
	where
	\begin{equation*}
		L_i^{\omega}(f_{ik}) = \sqrt{\sum_{j=1}^k X_{ij}^2}
			\xrightarrow[k \rightarrow \infty]{\text{a.s.}} \infty \,.
	\end{equation*}
	This can be shown via the Borel-Cantelli lemma,
	knowing the expected value \mbox{$\expect L_i^{\omega}(f_{ik})^2 = k$}
	and the growing variance of that expression.
	
	Specifically for approximation problems~%
	\mbox{$S = \App: \Hilbert \hookrightarrow G$},
	the functions~$g_i^{\omega}$, and therefore the output as well,
	are functions defined on the same domain as the input functions,
	but they are \emph{not} from the original Hilbert space~$\Hilbert$
	(for the same reasons
	that cause the functionals~\mbox{$L_i^{\omega}$} to be discontinuous).
	This shows in a very drastic way
	that the fundamental Monte Carlo approximation method is non-interpolatory.
	Yet the functions~$g_i^{\omega}$
	correspond to the information functionals~$L_i^{\omega}$,
	similarly to Hilbert space elements representing continuous linear functionals
	according to the Riesz representation theorem.
	
	Although the functionals~$L_i^{\omega}$ are almost surely discontinuous,
	for any fixed input~\mbox{$f \in \Hilbert$}
	the random information~\mbox{$L_i^{\omega}(f)$}
	is a standard Gaussian random variable with variance~$\|f\|_{\Hilbert}^2$,
	hence almost surely finite.
	To show almost sure convergence of the random series 
	by which we define~\mbox{$L_i^{\omega}(f)$}, however,
	requires some theory of Gaussian processes.
	The sequence of partial sums
	\mbox{$L_{i,m}^{\omega}(f)
						= \sum_{j=1}^m X_{ij} \, \langle \psi_j,f \rangle_{\Hilbert}$},
	is distributed like the collection of values~$B_{t_m}$
	of the standard Brownian motion~$B$
	at points~\mbox{$t_m := \sum_{j=1}^m \langle \psi_j,f \rangle_{\Hilbert}^2$}.
	Almost sure convergence of the random series follows from
	continuity of the Brownian motion and the limit
	\mbox{$t := \lim_{m \to \infty} t_m = \|f\|_{\Hilbert}^2$}.
	Of course, almost sure convergence implies convergence in distribution,
	so~$L_i^{\omega}(f)$ indeed is normally distributed.
	
	Moreover, the value of this functional is independent
	from the basis representation, so it is well-defined in a stochastic sense.
	Let~$(\varphi_j)_{j \in \N}$ be another orthonormal basis
	of the Hilbert space~$\Hilbert$.
	We compare the two Gaussian processes living on~$\N$,
	\begin{align*}
		L_m &:= \sum_{i=1}^m X_i \, \langle \psi_i, f \rangle_{\Hilbert} \,,
			\qquad X_i \stackrel{\text{iid}}{\sim} \Normal(0,1) \,, \\
		L_m' &:= \sum_{j=1}^m Y_i \, \langle \varphi_j, f \rangle_{\Hilbert} \,,
			\qquad Y_j := \sum_{i=1}^{\infty}
											X_i \, \langle \psi_i, \varphi_j \rangle_{\Hilbert}
								\stackrel{\text{iid}}{\sim} \Normal(0,1) \,.
	\end{align*}
	We aim to show for the limit $m\to\infty$
	that~$L_{\infty} = L_{\infty}'$ almost surely.
	Due to the almost sure convergence of each of the processes,
	if suffices to show that
	\begin{equation*}
		\expect (L_{\infty} - L_{\infty}')^2
			= \lim_{m \to \infty} \expect (L_m - L_m')^2
			= 0 \,.
	\end{equation*}
	The definition of the random variables $Y_j$ already
	relies on almost sure convergence of the Gaussian functional
	applied to~$\varphi_j$, hence,
	\begin{align*}
		\expect Y_j Y_k
			&\,=\, \lim_{m \to \infty}
					\expect
						\sum_{i=1}^m X_i \, \langle \psi_i, \varphi_j \rangle_{\Hilbert} \,
						\sum_{l=1}^m X_l \, \langle \psi_l, \varphi_k \rangle_{\Hilbert} \\
			&\,=\, \lim_{m \to \infty}
							\sum_{i=1}^m \langle \psi_i, \varphi_j \rangle_{\Hilbert}
													\langle \psi_i, \varphi_k \rangle_{\Hilbert} \\
			&\,=\, \langle \varphi_j,\varphi_k \rangle_{\Hilbert} \,,	
	\end{align*}
	so indeed, we have~\mbox{$Y_j \stackrel{\text{iid}}{\sim} \Normal(0,1)$}.
	A similar calculation shows
	\mbox{$\expect X_i Y_j = \langle \psi_i, \varphi_j \rangle_{\Hilbert}$}.
	From this we have
	\begin{align*}
		\expect (L_m - L_m')^2
			&\,=\, \sum_{i=1}^m \langle \psi,f \rangle_{\Hilbert}^2
							- 2 \sum_{i,j=1}^m \langle \psi_i,f \rangle_{\Hilbert}
																	\langle \psi_i, \varphi_j \rangle_{\Hilbert}
																	\langle \varphi_j, f \rangle_{\Hilbert}
							+ \sum_{j=1}^m \langle \varphi_j, f \rangle_{\Hilbert}^2 \\
			&\,=\, \left\| \sum_{i=1}^m \langle \psi_i,f \rangle_{\Hilbert} \, \psi_i
											- \sum_{j=1}^m
														\langle \varphi_j,f \rangle_{\Hilbert} \, \varphi_j
							\right\|_{\Hilbert}^2 \\
			&\xrightarrow[m \to \infty]{} \|f - f\|_{\Hilbert}^2
			\,=\, 0 \,.
	\end{align*}
	
	Since in addition, by assumption,
	the $g_i^{\omega}$~are almost surely defined,
	we have a method which is almost surely defined for any fixed~$f$.
	Even more, for any fixed~\mbox{$f \in \Hilbert$},
	the idealized algorithm~$A_n$
	can be approximated with almost sure convergence,
	\begin{equation*}
		A_{n,m}^{\omega}(f)
			\xrightarrow[m \rightarrow \infty]{\text{a.s.}}
				A_n^{\omega}(f) \,.
	\end{equation*}
	
	These considerations motivate to extend the class
	of admissible information functionals from continuous linear functionals
	to some class of ``stochastically bounded'' functionals~$\Lambda^{\rm stoch}$.
	Actually, this kind of stochastically defined functionals is quite common.
	For example, the problem of integrating $L_p$-functions by function values
	is only solvable in the randomized setting
	since in that case function evaluations are discontinuous.
	Compare also the example from Heinrich and Milla~\cite{HeM11}.
	
	Still, discontinuous functionals are not needed in principle in our situation.
	They are, however, the outcome of an idealization of methods~$A_{n,m}$.
	These use continuous functionals which
	behave nicely for any input function,
	despite the fact that they bear a huge norm on average.
	This is in accordance with the statements made in~\cite{HNW13},
	where the authors show for certain examples of continuous operators
	that discontinuous linear information
	is not more powerful than continuous information,
	supporting the believe that this should always be the case.
\end{remark}

\subsection{Reproducing kernel Hilbert spaces}
\label{sec:RKHS}

We summarize several facts about \emph{reproducing kernel Hilbert spaces} (RKHS)
that are necessary for the numerical analysis of approximation problems
\begin{equation*}
	\App: \Hilbert \hookrightarrow L_{\infty}(D) \,,
\end{equation*}
with~\mbox{$\Hilbert$} being a separable Hilbert space
of functions defined on a domain~\mbox{$D \subset \R^d$}.
For a general introduction to reproducing kernels,
refer to Aronszajn~\cite{Ar50}.
For an introduction with focus on the associated Gaussian field,
see Adler~\cite[Section~III.2]{Adl90}. 
The theory of RKHSs is a powerful concept
for the analysis of many other numerical settings,
e.g.~for certain average case problems
(see for instance Ritter~\cite[Chapter~III]{Rit00}),
or when information from function values is considered
(see Novak and Wo\'zniakowski~\cite{NW10,NW12} for a bunch of examples),
it also proves useful for statistical problems (see Wahba~\cite{Wah90}).


We assume function evaluations to be continuous on~\mbox{$\Hilbert$}.
Then by the Riesz representation theorem, for each~\mbox{$\vecx \in D$}
there exists a unique function~\mbox{$K_{\vecx}(\cdot) \in \Hilbert$}
such that for~\mbox{$f \in \Hilbert$} we have
\begin{equation*}
	f(\vecx) = \langle K_{\vecx},f \rangle_{\Hilbert} \,.
\end{equation*}
For~\mbox{$\vecx,\vecy \in D$} we define a \emph{symmetric} function
\begin{equation*}
	K(\vecx,\vecy)
		:= K_{\vecy}(\vecx)
		= \langle K_{\vecx},K_{\vecy} \rangle_{\Hilbert}
		= \langle K_{\vecy},K_{\vecx} \rangle_{\Hilbert}
		= K_{\vecx}(\vecy) = K(\vecy,\vecx) \,.
\end{equation*}
This function is called the \emph{reproducing kernel} of~\mbox{$\Hilbert$},
which is then called
\emph{reproducing kernel Hilbert space}~\mbox{$\Hilbert(K)$}.
There is a simple representation for the kernel,
namely, given an orthonormal basis~\mbox{$(\psi_i)_{i \in \N}$} 
of~\mbox{$\Hilbert(K)$}, we can write
\begin{equation} \label{eq:k=sumpsi}
	K(\vecx,\vecy)
		= \sum_{i = 1}^{\infty} \psi_i(\vecx) \, \psi_i(\vecy) \,.
\end{equation}


Knowing the kernel, it is easy to estimate the $\sup$-norm
of normalized functions~\mbox{$f \in \Hilbert$}.
Indeed, with \mbox{$\|f\|_{\Hilbert} = 1$} we get
\begin{multline} \label{eq:|f|sup,K}
	\|f\|_{\sup}
		= \sup_{\vecx \in D} |f(\vecx)|
		= \sup_{\vecx \in D} \langle K_{\vecx}, f \rangle_{\Hilbert}\\
		\leq \sup_{\vecx \in D} \|K_{\vecx}\|_{\Hilbert}
		= \sup_{\vecx \in D} \sqrt{\langle K_{\vecx}, K_{\vecx} \rangle_{\Hilbert}}
		= \sup_{\vecx \in D} \sqrt{K(\vecx,\vecx)} \,.
\end{multline}
This is the initial error for the uniform approximation problem.
By the Cauchy-Schwarz inequality, the $\sup$-norm of $K$
is determined by values~\mbox{$K(\vecx,\vecx)$}
on the diagonal of~\mbox{$D \times D$},
\begin{equation*}
	\sup_{\vecx,\vecy \in D} |K(\vecx,\vecy)|
		= \sup_{\vecx \in D} K(\vecx,\vecx) \,.
\end{equation*}
Therefore from now on we assume the kernel~$K$ to be bounded.


We consider the \emph{canonical metric}
\mbox{$d_K: D \times D \rightarrow [0,\infty)$} associated to~$\Hilbert(K)$,
\begin{equation*}
	d_K(\vecx,\vecy)	
		:= \|K_{\vecx} - K_{\vecy}\|_{\Hilbert}
		= \sqrt{K(\vecx,\vecx) - 2 \, K(\vecx,\vecy) + K(\vecy,\vecy)} \,.
\end{equation*}
(If~$d_K(\vecx,\vecy) = 0$ for some distinct~{$\vecx \not= \vecy$},
we only have a semimetric. Then we still obtain
a metric for the set of equivalence classes of points that are at distance~$0$.)
Functions~\mbox{$f \in \Hilbert$} are Lipschitz continuous
with Lipschitz constant~\mbox{$\|f\|_{\Hilbert}$}
with respect to the canonical metric,
\begin{equation*}
	|f(\vecx) - f(\vecy)|
		= |\langle K_{\vecx} - K_{\vecy}, f \rangle_{\Hilbert}|
		\leq \|K_{\vecx} - K_{\vecy} \|_{\Hilbert} \, \|f\|_{\Hilbert}
		= \|f\|_{\Hilbert} \, d_K(\vecx,\vecy) \,.
\end{equation*}
Hence functions from~$\Hilbert$
are continuous with respect to any metric~$\delta$ on~$D$
that is topologically equivalent to the canonical metric~$d_K$.
Since we assume~$K$ to be bounded,
the domain~$D$ is bounded with respect to~$d_K$,
\begin{equation*}
	\diam(D) = \sup_{\vecx,\vecy \in D} d_K(\vecx,\vecy)
		\leq 2 \, \sup_{\vecx \in D} \sqrt{K(\vecx,\vecx)}\,.
\end{equation*}


Tensor products constitute a common way in IBC to define multivariate problems,
compare for instance Novak and Wo\'zniakowski~\cite[Section~5.2]{NW08},
or Ritter~\cite[Section~VI.2]{Rit00},
find many more examples in Novak and Wo\'zniakowski~\cite{NW10,NW12}.
Let~\mbox{$\Hilbert(K_1)$} and \mbox{$\Hilbert(K_2)$}
be reproducing kernel Hilbert spaces
defined on~$D_1$ and $D_2$, respectively.
Let~\mbox{$(\varphi_i)_{i \in \N}$} and~\mbox{$(\psi_j)_{j \in \N}$}
be corresponding orthonormal bases.
Then the tensor product space~%
\mbox{$\Hilbert(K_1) \otimes \Hilbert(K_2)$}
is the Hilbert space with orthonormal basis~%
\mbox{$(\varphi_i \otimes \psi_j)_{i,j \in \N}$}.
Here, \mbox{$f_1 \otimes f_2$} denotes the tensor product of functions~%
\mbox{$f_1 \in \Hilbert(K_1)$} and \mbox{$f_2 \in \Hilbert(K_2)$},
\begin{equation*}
	[f_1 \otimes f_2](\vecx_1,\vecx_2)
	:= f_1(\vecx_1) \, f_2(\vecx_2)\,,
	\quad\text{defined for $(\vecx_1,\vecx_2) \in D_1 \times D_2$.}
\end{equation*}
With another tensor product function~\mbox{$g_1 \otimes g_2$} of this sort,
one readily obtains
\begin{equation*}
	\langle f_1 \otimes f_2,
	g_1 \otimes g_2
	\rangle_{\Hilbert1 \otimes \Hilbert_2}
	= \langle f_1,g_1 \rangle_{\Hilbert_1}
	\, \langle f_2,g_2 \rangle_{\Hilbert_2} \,.
\end{equation*}
Using the representation~\eqref{eq:k=sumpsi}, 
it is easy to see that the reproducing kernel~$K$ of the new space
is the tensor product of the kernels of the original spaces,
\begin{equation*}
	K((\vecx_1,\vecx_2),(\vecy_1,\vecy_2))
	:= K_1(\vecx_1,\vecy_1) \, K_2(\vecx_2,\vecy_2) \,,
\end{equation*}
where~\mbox{$(\vecx_1,\vecx_2),(\vecy_1,\vecy_2) \in D_1 \times D_2$}.

\subsection{Expected maximum of zero-mean Gaussian fields}
\label{sec:E|Psi|_sup}

We discuss zero-mean Gaussian fields and their connection to
reproducing kernel Hilbert spaces.
All results collected here can be found in the notes by Adler~\cite{Adl90}.
For some results, Lifshits~\cite{Lif95} or Ledoux and Talagrand~\cite{LT91}
will also be good references.


Let \mbox{$(\psi_i)_{i \in \N}$} be an orthonormal basis
of a reproducing kernel Hilbert space~\mbox{$\Hilbert(K)$}.
Then the pointwise definition
\begin{equation} \label{eq:Psi=sum(X*psi)}
	\Psi_{\vecx} := \sum_{i=1}^{\infty} X_i \, \psi_i(\vecx) \,,
\end{equation}
with $X_i$ being i.i.d.\ standard Gaussian random variables,
produces a random function~$\Psi$ defined on~$D$
where the covariance function is the kernel~$K$,
\begin{equation*}
	\Cov(\Psi_{\vecx},\Psi_{\vecy})
		= \sum_{i,j=1}^{\infty}
				(\expect X_i \, X_j)
					\, \psi_i(\vecx) \, \psi_j(\vecy)
		= \sum_{i = 1}^{\infty}
				\psi_i(\vecx) \, \psi_i(\vecy)
		= K(\vecx,\vecy) \,.
\end{equation*}
Definition~\eqref{eq:Psi=sum(X*psi)} is actually the outcome
of applying the Gaussian functional from \remref{rem:Lstoch}
to the kernel functions~\mbox{$K_{\vecx} \in \Hilbert(K)$}.
It turns out that this definition is pointwise almost surely convergent
and independent from the choice of the orthonormal basis in~$\Hilbert(K)$.


Note that for the \emph{canonical metric}~\mbox{$d_K$}
associated to the reproducing kernel Hilbert space~\mbox{$\Hilbert(K)$}
we have the alternative representation
\begin{equation*}
	d_K(\vecx,\vecy)
		= \sqrt{\expect (\Psi_{\vecx} - \Psi_{\vecy})^2} \,.
\end{equation*}
From now on we assume that the domain~$D$ is \emph{totally bounded},
that is, for any~\mbox{$r > 0$}
the set~$D$ can be covered by finitely many balls with radius~$r$.
(If~$D$ is complete with respect to~$d_K$,
this implies compactness of~$D$.
Conversely, compactness of a metric space implies total boundedness.)
In that case,
for Gaussian fields,
continuity of~$\Psi$ with respect to the canonical metric~$d_K$
is equivalent to boundedness, see Adler~\cite[Theorem~4.16]{Adl90}.
Further, if~$\Psi$ is almost surely continuous,
then we have almost sure uniform convergence
of the series~\eqref{eq:Psi=sum(X*psi)} on~$D$
(that is, with respect to the $L_{\infty}$-norm),
see Adler~\cite[Theorem~3.8]{Adl90}.

One method for estimating the maximum of a Gaussian field
is based on \emph{metric entropy}. 
For~\mbox{$r > 0$}, let \mbox{$N(r) = N(r,D,d_K)$},
denote the minimal number of $d_K$-balls with radius~$r$
needed to cover~$D$.
The function~\mbox{$H(r) := \log N(r)$}
is called the \emph{(metric) entropy} of~$D$.
The following inequality is based on this quantity,
it goes back to Dudley 1973~\cite[Theorem~2.1]{Dud73}.
\begin{proposition}[Dudley] \label{prop:Dudley}
	There exists a universal constant~\mbox{$C_{\textup{Dudley}} > 0$} such that
	\begin{equation*}
		\expect \sup_{\vecx \in D} \Psi_{\vecx}
			\leq C_{\textup{Dudley}}
							\, \int_0^{\infty}
											\sqrt{\log N(r)}
										\dint r \,.
	\end{equation*}
\end{proposition}
For a direct proof with explicit numerical bound~%
\mbox{$C_{\textup{Dudley}} \leq 4 \sqrt{2}$},
see Lifshits~\cite[Section~14, Theorem~1]{Lif95}.
In the book of Adler~\cite[Corollary~4.15]{Adl90}
Dudley's inequality was derived from Fernique's estimate,
which in its own right is a suitable tool for estimating
the expected maximum of Gaussian fields,
see Adler~\cite[Theorem~4.1]{Adl90}.

It is well known that covering numbers~$N(r)$ can be estimated from above
via packing numbers and volume estimates of balls.
Let \mbox{$B_K(\vecx,r)$} denote the closed $d_K$-ball 
around~\mbox{$\vecx \in D$}.
For any probability measure~$\mu$ on $D$ we have
\begin{equation*}
	N(r) \leq \sup_{\vecx \in D} 1/\mu(B_K(\vecx,r/2)) \,.
\end{equation*} 
Plugging this into \propref{prop:Dudley} (Dudley),
after a change of variables we can state
\begin{equation} \label{eq:Dudley2}
	\expect \sup_{\vecx \in D} \Psi_{\vecx}
		\leq 2 \, C_{\textup{Dudley}}
						\, \int_0^{\infty}
										\sup_{\vecx \in D} \sqrt{\log (1/\mu(B_{K}(\vecx,r)))}
									\dint r \,.
\end{equation}

Since we are interested in the expected $\sup$-norm of~$\Psi$,
we also need the following elementary lemma,
compare Adler~\cite[Lemma~3.1]{Adl90}.
\begin{lemma} \label{lem:E|X|<init+EsupX}
	For the Gaussian field~$\Psi$ with covariance function~$K$,
	we have
	\begin{equation*}
		\expect \|\Psi\|_{\infty}
			\leq \sqrt{\frac{2}{\pi}}
							\, \inf_{\vecx \in D} \sqrt{K(\vecx,\vecx)}
						+ 2 \expect \sup_{\vecx \in D} \Psi_{\vecx} \,.
	\end{equation*}
\end{lemma}

\subsection{A lower bound for deterministic approximation}
\label{sec:HilbertWorLB}

Osipenko and Parfenov 1995~\cite{OP95},
Kuo, Wasilkowski, and Wo\'zniakowski 2008~\cite{KWW08},
and Cobos, K\"uhn, and Sickel 2016~\cite{CKS16},
independently from each other found similar approaches
to relate the error of~$L_{\infty}$-approximation
to $L_2$-approximation.
Let $\rho$ be a measure on~$D$
(defined for Borel sets in~$D$, with respect to the canonical metric~$d_K$).
Recall that~\mbox{$L_p(\rho)$} denotes the space of (equivalence classes of)
measurable functions defined on~$D$ and
bounded in the norm
\begin{equation*}
	\|f\|_{L_p(\rho)}
		:= \begin{cases}
					\left(\int_D f^p \dint \rho\right)^{1/p}
						\quad&\text{for $1 \leq p < \infty$,}\\
					\esssup_{D,\rho} |f| 
						\quad&\text{for $p = \infty$,}
				\end{cases}
\end{equation*}
where
\mbox{$\esssup_{D,\rho} |f|
				:= \sup \{\lambda \in \R \mid
									\rho\{\vecx \in D : |f(\vecx)| \geq \lambda\} > 0\}$}.
For continuous functions it makes sense to consider the supremum norm
\begin{equation*}
	\|f\|_{\sup} := \sup_{\vecx \in D} |f(\vecx)| \geq \|f\|_{L_{\infty}(\rho)} \,.
\end{equation*}
Later, when the supremum norm and the $L_{\infty}$-norm coincide,
we will only write~\mbox{$\|\cdot\|_{\infty}$}.

The following version of a deterministic lower bound
is close to the formulation of
Osipenko and Parfenov~\cite[Theorem~3]{OP95},
also Cobos et al.~\cite[Lemma~3.3]{CKS16},
however, it is essentially contained in Kuo et al.~\cite{KWW08}
as well.
(Kuo et al.\ work with eigenvalues of an integral operator
defined via the kernel function~$K$.
These eigenvalues are the squared singular values,
which in turn we prefer to use here.)
\begin{proposition} \label{prop:H->L_inf/L_2}
	Let~$\rho$ be a probability measure on a domain~$D$,
	and consider a separable reproducing kernel Hilbert space~%
	\mbox{$\Hilbert = \Hilbert(K)$}
	which is compactly embedded into~\mbox{$L_{\infty}(\rho)$}.
	The embedding~\mbox{$\Hilbert \hookrightarrow L_2(\rho)$}
	is compact as well, and a singular value decomposition exists.
	This means, there is an orthonormal basis~\mbox{$(\psi_k)_{k=1}^M$}
	(with~\mbox{$M \in \N \cup \{\infty\}$}) of~$\Hilbert$
	which is also orthogonal in~\mbox{$L_2(\rho)$},
	and the corresponding
	singular values~\mbox{$\sigma_k := \|\psi_k\|_{L_2(\rho)}$},
	for~\mbox{$k < M+1$},
	are in decaying order~\mbox{$\sigma_1 \geq \sigma_2 \geq \ldots \geq 0$}.
	
	Then we have
	\begin{equation*}
		e^{\deter}(n,\Hilbert \hookrightarrow L_{\infty}(\rho))
			\geq \sqrt{\sum_{k=n+1}^{\infty}
									\sigma_k^2
								} \,.
	\end{equation*}
\end{proposition}

\section{Breaking the curse for periodic functions}
\label{sec:HilbertPeriodic}

\subsection{The setting}


We study the~$L_{\infty}$-approximation of Hilbert space functions
defined on the $d$-dimen\-sional torus~$\Torus^d$,
compare the notation in Cobos et al.~\cite{CKS16} (with slight modifications).
The Hilbert spaces we consider will be unweighted tensor product spaces.

A few words on the domain.
The one-dimensional torus~\mbox{$\Torus := \R \modulo \Z \equiv [0,1)$}
can be identified with the unit interval tying the endpoints together.
A natural way to define a metric on~$\Torus$ is
\begin{equation*}
	d_{\Torus}(x,y) := \min_{k \in\{-1,0,1\}} |x-y+k| \,,
	\quad \text{for\, $x,y \in [0,1)$.}
\end{equation*}
This is the length of the shortest connection between two points
along a closed curve of length~$1$.
For the~$d$-dimensional torus we consider $\ell_p$-like (quasi)-metrics
\begin{equation*}
	d_p(\vecx,\vecy) := \left(\sum_{j=1}^d d_{\Torus}(x_j,y_j)^p\right)^{1/p} \,,
\end{equation*}
where~$p \in (0,\infty]$.
Smoothness and continuity are to be defined with respect to the $d_p$-metric
with~$p \in [1,\infty]$.

We start with a basis representation of spaces under consideration.
First, for~$d=1$, the Fourier system
\begin{equation*}
	\{\varphi_0 := 1,\,
		\varphi_{-k} := \sqrt{2} \, \sin(2 \pi k \, \cdot),\,
		\varphi_k := \sqrt{2} \, \cos(2 \pi k \, \cdot)
	\}_{k \in \N}
\end{equation*}
is an orthonormal basis for~\mbox{$L_2(\Torus) = L_2([0,1))$}.
We consider Hilbert spaces where these functions are still orthogonal.
Namely, let \mbox{$\Hilbert_{\veclambda}(\Torus)$}
denote the Hilbert space for which the system
\begin{equation*}
	\left\{\psi_0 := \lambda_0,\,
				\psi_{-k} := \lambda_k \, \sin(2 \pi k \, \cdot),\,
				\psi_k := \lambda_k \, \cos(2 \pi k \, \cdot)
	\right\}_{k \in \N} \,,
\end{equation*}
is an orthonormal basis. 
Here, \mbox{$\veclambda = (\lambda_k)_{k \in \N_0} \subset (0,\infty)$}
indicates the importance of the different frequencies.
Now, for general~\mbox{$d \in \N$},
we consider the \emph{unweighted}
tensor product space~\mbox{$\Hilbert_{\veclambda}(\Torus^d)$}
with the tensor product
orthonormal basis~\mbox{$\{\psi_{\veck}\}_{\veck \in \Z^d}$},
\begin{equation} \label{eq:periodicpsi}
	\psi_{\veck}(\vecx) := \prod_{j=1}^d \psi_{k_j}(x_j) \,. 
\end{equation}
(For \emph{weighted} tensor product spaces one would take different
values for~$\veclambda$ for different dimensions~\mbox{$j=1,\ldots,d$},
compare Cobos et al.~\cite{CKS16}, or Kuo et al.~\cite{KWW08}.
In contrast,
\emph{unweighted} means that all coordinates are equally important.)
Analogously, we write~\mbox{$\{\varphi_{\veck}\}_{\veck \in \Z^d}$}
for the Fourier basis of~\mbox{$L_2(\Torus^d)$}.

For a suitable choice of the~$\lambda_k$,
we have the one-dimensional reproducing kernel
\begin{align}
	K_{\veclambda}(x,y)
		&:= \lambda_0^2
				+ \sum_{k=1}^{\infty}
						\lambda_k^2
							\, [\cos(2 \pi k \, x) \, \cos (2 \pi k \, y)
									+ \sin (2 \pi k \, x) \, \sin (2 \pi k \, y)] 
			\nonumber\\
		&= \sum_{k=0}^{\infty} \lambda_k^2 \, \cos (2 \pi k\,(x-y))
		\label{eq:K_la} \,,
\end{align}
for general dimensions~\mbox{$d \in \N$} we obtain the product kernel
\begin{equation*}
	K_{\veclambda}^d(\vecx,\vecy)
		:= \prod_{j=1}^d K_{\veclambda}(x_j,y_j) \,.
\end{equation*}
From this we derive the initial error,
\begin{equation*}
	e(0,\Hilbert_{\veclambda}(\Torus^d) \embed L_{\infty}(\Torus^d))
		= \sup_{\vecx \in \Torus^d} \sqrt{K_{\veclambda}^d(\vecx,\vecx)}
		= \left( \sum_{k=0}^{\infty} \lambda_k^2 \right)^{d/2} \,.
\end{equation*}
The condition~\mbox{$\sum_{k=0}^{\infty} \lambda_k^2 < \infty$}
is necessary and sufficient for the existence of a reproducing kernel
and for the embedding
\mbox{$\Hilbert_{\veclambda}(\Torus^d) \hookrightarrow L_{\infty}$}
to be compact,
see Cobos et al.~\cite[Theorem~3.1]{CKS16}
with an extended list of equivalent properties.
We will assume \mbox{$\sum_{k=0}^{\infty} \lambda_k^2 = 1$},
then the initial error is~$1$, independently from the dimension~$d$.
This makes the family
\mbox{$(\Hilbert_{\veclambda}(\Torus^d) \embed L_{\infty}(\Torus^d))_{d \in \N}$} of $d$-dependent problems
comparable, we say, ``the problem is properly normalized''.

Note that under this last assumption,
functions~\mbox{$f \in \Hilbert_{\veclambda}(\Torus^d)$}
can be identified with
functions~\mbox{$\tilde{f} \in \Hilbert_{\veclambda}(\Torus^{d+1})$},
\begin{equation*}
	\tilde{f}(x_1,\ldots,x_{d+1})
		:= f(x_1,\ldots,x_d) \, K_{\veclambda}(0,x_{d+1}) \,,
\end{equation*}
the $\Hilbert_{\veclambda}$- and the $L_{\infty}$-norms coincide,
the maximum values of the function being attained for~\mbox{$x_{d+1} = 0$}.
So indeed,
the problems of lower dimensions are embedded
in the problems of higher dimensions,
yet~$\tilde{f}$ is a bit lopsided in the redundant variable.


In particular we consider unweighted \emph{Korobov spaces}.
Within the above framework, these are spaces~%
\mbox{$\Hilbert_r^{\Korobov}(\Torus^d)
				:= \Hilbert_{\veclambda}(\Torus^d)$}
with~\mbox{$\lambda_0 = \sqrt{\beta_0}$}
and~\mbox{$\lambda_k = \sqrt{\beta_1} \, k^{-r}$}
for~\mbox{$k \in \N$}, where~$\beta_0,\beta_1 > 0$.
For integers~\mbox{$r \in \N$},
the Korobov space norm can be given in a natural way
in terms of weak partial derivatives (instead of Fourier coefficients),
in the one-dimensional case we have
\begin{equation*}
	\|f\|_{\Hilbert_r^{\Korobov}(\Torus)}^2
		= \beta_0^{-1} \left| \int_{\Torus} f(x) \dint x \right|^2
			+ \beta_1^{-1} \, (2\pi)^{-2r} \, \|f^{(r)}\|_2^2 \,.
\end{equation*}
The $d$-dimensional case is a bit more complicated,
in a squeezed way, the norm is
\begin{equation*}
	\|f\|_{\Hilbert_r^{\Korobov}(\Torus^d)}^2
		= \sum_{J \subseteq [d]}
				\beta_0^{-(d - \#J)} \, (\beta_1^{-1} \, (2\pi)^{-2r})^{\#J}
					\, \left\| \int_{\Torus^{[d] \setminus J}}
												\biggl(\prod_{j \in J} \partial_j^r\biggr)
													f(\vecx) \dint \vecx_{[d] \setminus J}
						\right\|_{L_2(\Torus^J)}^2 \,,
\end{equation*}
see Novak and Wo\'zniakowski~\cite[Section~A.1]{NW08} for details
on the derivation of this representation of the norm.
There one can also find some information
on the historical background concerning these spaces.
It should be pointed out that in the same book tractability
for $L_2$-approximation of Korobov functions based on linear information
has been studied~\cite[pp.~191--193]{NW08},
in that case randomization does not help a lot.

The condition~\mbox{$r > 1/2$}
is necessary and sufficient for the existence of a reproducing kernel
(and the embedding~%
\mbox{$\Hilbert_r^{\Korobov}(\Torus^d)
					\hookrightarrow L_{\infty}(\Torus^d)$}
to be compact),
then
\begin{equation*}
	\sum_{k=1}^{\infty} \lambda_k^2
		= \beta_1 \, \sum_{k=1}^{\infty} k^{-2 r}
		= \beta_1 \, \zeta(2 r)
\end{equation*}
with the Riemann zeta function~$\zeta$.
Assuming
\begin{equation}\label{eq:beta12init}
	\beta_0 + \beta_1 \, \zeta(2 r) = 1 \,,
\end{equation}
the initial error will be constant~$1$ in all dimensions.

\subsection{The curse for deterministic approximation}

\begin{theorem} \label{thm:curseperiodic}
	Suppose~\mbox{$0 \leq \lambda_0 < 1$}
	and~\mbox{$\sum_{k=0}^{\infty} \lambda_k^2 = 1$}
	for non-negative~$\lambda_k$.
	Then the approximation problem
	\begin{equation*}
		\App: \Hilbert_{\veclambda}(\Torus^d) \hookrightarrow L_{\infty}(\Torus^d)
	\end{equation*}
	suffers from the curse of dimensionality in the deterministic setting.
	
	In detail, while the initial error is constant~$1$, we have
	\begin{equation*}
		e^{\deter}(n,\Hilbert_{\veclambda}(\Torus^d)
									\hookrightarrow L_{\infty}(\Torus^d))
			\geq \sqrt{(1 - n \, \beta^d)_+} \,,
	\end{equation*}
	where~\mbox{$\beta := \sup\{\lambda_0^2,\,
															\lambda_k^2/2
														\}_{k \in \N} \in (0,1)$}.
	In other words,
	for~\mbox{$\eps \in (0,1)$} we have the complexity bound
	\begin{equation*}
		n^{\deter}(\eps,\Hilbert_{\veclambda}(\Torus^d)
									\hookrightarrow L_{\infty}(\Torus^d))
			\geq \beta^{-d}(1-\eps)^2 \,.
	\end{equation*}
	
	The curse and this complexity bound hold in particular
	for Korobov spaces~$\Hilbert_r^{\Korobov}(\Torus^d)$
	with \mbox{$\beta = \sup\{\beta_0,\,\beta_1/2\}$}.
\end{theorem}
\begin{proof}
	Following \propref{prop:H->L_inf/L_2}, we study the singular values
	of~\mbox{$\Hilbert_{\veclambda}(\Torus^d) \hookrightarrow L_2(\Torus^d)$}.
	Essentially, this can be traced back to the one-dimensional case,
	\begin{equation*}
		\psi_k = \sigma_k \varphi_k
		\quad \text{for\, $k \in \Z$,}
	\end{equation*}
	where~\mbox{$\sigma_0 = \lambda_0$}
	and \mbox{$\sigma_k = \sigma_{-k} = \lambda_k / \sqrt{2}$} for~{$k \in \N$}
	denote the unordered singular values
	of~\mbox{$\Hilbert_{\veclambda}(\Torus) \hookrightarrow L_2(\Torus)$}.
	In the multi-dimensional case we have
	\begin{equation*}
		\psi_{\veck} = \sigma_{\veck} \varphi_{\veck}
		\quad \text{for\, $\veck \in \Z^d$,}
	\end{equation*}
	with the unordered singular values~%
	\mbox{$\sigma_{\veck} = \prod_{j=1}^d \sigma_{k_j}$},
	in particular
	\begin{equation*}
		\sigma_{\veck}^2
			\leq \left(\sup_{k' \in \Z} \sigma_{k'}^2\right)^d
			= \left(\sup\{\lambda_0^2, \, \lambda_{k'}^2 / 2\}_{k' \in \N}\right)^d
			= \beta^d \,.
	\end{equation*}
	On the other hand,
	\begin{equation*}
		\sum_{\veck \in \Z^d} \sigma_{\veck}^2
			= \left(\sum_{k' \in \Z} \sigma_{k'}^2 \right)^d
			= \left(\sum_{k' = 0}^{\infty} \lambda_{k'}^2 \right)^d
			= 1 \,.
	\end{equation*}
	So for any index set~\mbox{$I \subset \Z^d$} of size~\mbox{$\# I = n$},
	we have
	\begin{equation*}
		\sqrt{\sum_{\veck \in \Z^d \setminus I} \sigma_{\veck}^2}
			\geq \sqrt{(1 - n \beta^d)_+} \,.
	\end{equation*}
	By \propref{prop:H->L_inf/L_2}, this proves the lower bound.
\end{proof}

\begin{remark}[Structure of optimal deterministic methods]
	Within the above proof we applied \propref{prop:H->L_inf/L_2}
	with~$\rho$ being the uniform distribution on~$\Torus^d$.
	If we consider complex-valued Hilbert spaces,
	theoretically this approach will always give sharp lower bounds,
	with algorithms projecting onto $n$-dimensional spans
	of a subset of the complex orthonormal basis
	\mbox{$\{\euler^{2\pi\imag(\veck,\cdot)}\}_{\veck \in \Z^d}$},
	see Cobos et al.~\cite[Theorem~3.4]{CKS16}.
	Here, $(\veck,\vecx)$ is the standard scalar product in~$\R^d$.
	The problem of this for the real-valued setting is
	that, for~\mbox{$\veck \in \Z^d \setminus \zeros$},
	the Fourier coefficients belonging to~$\euler^{2\pi (\veck,\cdot)}$
	in general will be complex, even for real-valued functions.
	Alternatively, we could consider a real-valued variant of a
	$d$-dimensional Fourier basis,
	containing functions~$\sqrt{2} \sin(2\pi (\veck,\cdot))$
	and~$\sqrt{2} \cos(2\pi (\veck,\cdot))$.
	A real-valued algorithm with matching lower bounds
	will be a projection onto an $n$-dimensional subspace
	spanned by a selection
	of such pairs, and -- as the case may be -- the constant function.
	Hence, in the real-valued setting upper and lower bounds almost match,
	with gaps for the $n$-th error bounded by the $n$-th singular value.
	The above result on the curse of dimensionality is much rougher,
	employing only a general upper bound which is valid for all singular values.
\end{remark}

\subsection{Polynomial Tractability via Monte Carlo}

The key tool for understanding the canonical metric for tensor product kernels
is an estimate on the shape of the kernel in terms of some local polynomial decay.
We have the following general result.

\begin{proposition} \label{prop:MCUBperiodic}
	Consider the uniform approximation problem
	\begin{equation*}
		\App: \Hilbert(K_d) \hookrightarrow L_{\infty}(\Torus^d)
	\end{equation*}
	where~$\Hilbert(K_d)$ is a reproducing kernel Hilbert space on
	the $d$-dimensional torus~$\Torus^d$ with the following properties:
	\begin{enumerate}[(i)]
		\item \label{enum:MCUBperiodic,unweighted}
			$K_d$ is the unweighted product kernel built
			from the one-dimensional case,\\
			this means
			\mbox{$K_d(\vecx,\vecy) := \prod_{j=1}^d K_1(x_j,y_j)$}
			for~\mbox{$\vecx,\vecy \in \Torus^d$}.
		\item \label{enum:MCUBperiodic,init}
			\mbox{$K_1(x,x) = 1$} for all~$x \in \Torus$.\\
			(Consequently, \mbox{$K_d(\vecx,\vecx) = 1$}
			for all~\mbox{$\vecx \in \Torus^d$},
			in particular the initial error is constant~$1$.)
		\item \label{enum:MCUBperiodic,local}
			The kernel function can be locally estimated from below
			with a polynomial decay,
			that is,
			there exist \mbox{$\alpha > 0$}, \mbox{$p \in (0,1]$},
			and~\mbox{$0 < R_0 \leq 1/2$},
			such that
			\begin{equation*}
				K_1(x,y) \geq 1 - \alpha \, d_{\Torus}(x,y)^p
				\quad
				\text{for \mbox{$x,y \in \Torus$} with~\mbox{$d_{\Torus}(x,y) \leq R_0$}.}
			\end{equation*}
			(Hence \mbox{$K_d(\vecx,\vecy)
															\geq 1 - \alpha \, d_p(\vecx,\vecy)^p$}
			\quad for \mbox{$\max_j d_{\Torus}(x_j,y_j) \leq R_0$}.)
		\end{enumerate}
	Then the problem is polynomially tractable
	in the randomized setting with general linear information,
	in detail,
	\begin{equation*}
		n^{\ran}(\eps,\Hilbert(K_d) \hookrightarrow L_{\infty}(\Torus^d))
			\leq C(p) \, (1 + \alpha - \log 2 R_0) \, \frac{d \, (1 + \log d)}{\eps^2} \,,
	\end{equation*}
	with a constant~\mbox{$C(p) > 0$} which depends only on~$p$.
\end{proposition}
\begin{proof}
	We are going to apply \propref{prop:Dudley} (Dudley) in order
	to estimate the expected maximum norm of the Gaussian field~$\Psi$
	associated with the reproducing kernel~$K_d$.
	We will work with volume estimates,
	where $\mu$ shall be the uniform distribution on~$\Torus^d$,
	this is the Lebesgue measure on~\mbox{$[0,1)^d$}.
	
	Suppose \mbox{$\max_j d_{\Torus}(x_j,y_j) \leq R_0$},
	then for the canonical metric in the $d$-dimensional case we have
	\begin{align*}
		d_K(\vecx,\vecy)^2
			&= K_d(\vecx,\vecx) + K_d(\vecy,\vecy) - 2 \, K_d(\vecx,\vecy) \\
			&\leq 2 \alpha \, d_p(\vecx,\vecy)^p \,.
	\end{align*}
	By this, we have the inclusion
	\begin{equation*}
		B_K(\vecx,r) \supseteq
		B_{\Torus,p}\left(\vecx,\,
											\left(\frac{r^2}{2 \alpha}\right)^{1/p}
								\right) \,,
	\end{equation*}
	where~\mbox{$B_{\Torus,p}(\vecx,R)$} denotes
	the $d_p$-ball of radius~$R$ around~$\vecx \in \Torus^d$,
	and \mbox{$B_K(\vecx,r)$} is the ball of radius~$r$ in the canonical metric
	associated with~$K_d$.
	Hence
	\begin{equation*}
		\mu(B_K(\vecx,r))
			\geq \mu\left(B_{\Torus,p}\left(\vecx,\,
																			\left(\frac{r^2}{2 \alpha}\right)^{1/p}
																\right)
							\right) \,.
	\end{equation*}
	We distinguish three cases:
	\begin{itemize}
		\item For~\mbox{$0 \leq R \leq R_0 \leq 1/2$},
			the $\mu$-volume \mbox{$\mu(B_{\Torus,p}(\vecx,R))$} of the $d_p$-ball
			is the volume~\mbox{$\Vol(R \, B_p^d)$}
			of an $\ell_p$-ball in~$\R^d$ with radius~$R$,
			so with Stirling's formula,
			\begin{align}
				\log(1/\mu(B_{\Torus,p}(\vecx,R)))
					&= \log(1/\Vol(R \, B_p^d)) \nonumber\\
					&= \log \Gamma\left(\frac{d}{p} + 1\right)
								- d \, \log \left[2 R \, \Gamma\left(\frac{1}{p} + 1\right)
														\right]
							\nonumber\\
				[\textstyle \Gamma\bigl(\frac{1}{p} + 1\bigr) \geq 1]\qquad
					&\leq C_1 \, \frac{d}{p} \, \left(1 + \log \frac{d}{p}\right)
							\, (1 - \log 2 R)
					\nonumber \\
					&\leq C_2(p) \, d \, (1 + \log d) \, (1 - \log (2 R)^{p/2})
					\label{eq:log(1/mu(B2))}\,.
			\end{align}
			Such an estimate can be used
			for~\mbox{$0 \leq r
									 \leq \sqrt{2\alpha \, R_0^p}
									 \leq \sqrt{2^{1-p}\,\alpha}$}.
			Concerning the behaviour of the constant we can state
			\mbox{$C_2(p) \preceq p^{-2} (1 + \log p^{-1})$}.
		\item For~\mbox{$R > R_0$},
			the $\mu$-volume of~\mbox{$B_{\Torus,p}(\vecx,R)$}
			can be estimated from below by the $\mu$-volume
			of an $\ell_p$-ball with radius~\mbox{$R_0 \leq 1/2$},
			we have
			\begin{equation*}
				\log(1/\mu(B_{\Torus,p}(\vecx,R)))
					\leq C_2(p) \, d \, (1 + \log d) \, (1 - \log 2 R)
			\end{equation*}
			We will use this to cover the case~\mbox{$\sqrt{2\alpha \, R_0^p} < r < 2$}.
		\item For~\mbox{$r \geq 2$},
			we know~\mbox{$B_K(\vecx,r) = \Torus^d$} with $\mu$-volume~$1$
			since~\mbox{$d_K(\vecx,\vecy) \leq 2$}.
			In this case the term~$\log(1/\mu(B_K(\vecx,r)))$
			vanishes.
	\end{itemize}
	Combining these cases, we can estimate
	\begin{align*}
		\int_0^{\infty} \sqrt{\log(1/\mu(B_K(\vecx,r)))} \dint r
			&\leq \int_0^{\sqrt{2^{1-p}\,\alpha}}
							\sqrt{\log \left(1 \middle/
																		\Vol\left(\left(\frac{r^2}{2 \alpha}
																							\right)^{1/p}
																								\, B_p^d
																				\right)
													\right)
										} \,
								\dint r \\
			&\qquad + \left(2 - \sqrt{2\alpha \, R_0^p}\right)_+
								\, \sqrt{\log\left(1 \middle/
																			\Vol
																				\left(R_0 \, B_p^d
																				\right)
															\right)
													} \\
			&\stackrel{\text{\eqref{eq:log(1/mu(B2))}}}{\leq}
				\sqrt{C_2(p) \, d \, (1 + \log d)} \\
			&\qquad\qquad \Biggl(\int_0^{\sqrt{2^{1-p}\,\alpha}}
															\sqrt{1 - \log \frac{r}{\sqrt{2^{1-p}\,\alpha}}}
														\dint r \\
			&\qquad\qquad\phantom{\Biggl(}
														+\, 2 \, \sqrt{1 - \log 2 R_0}
										\Biggr) \\
			&= C_3(p) \, \sqrt{d \, (1 + \log d)}
							\,\bigl(1 + \sqrt{\alpha} + \sqrt{- \log 2 R_0}\bigr) \,.
	\end{align*}
	Here, the last integral can be transformed into a familiar integral
	by the substitution~\mbox{$s^2/2 = 1 - \log (r/\sqrt{2^{1-p}\,\alpha})$},
	\begin{equation*}
		\int_0^{\sqrt{2^{1-p} \alpha}}
						\sqrt{1 - \log \frac{r}{\sqrt{2^{1-p} \, \alpha}}}
					\dint r
			= \euler \, 2^{-p} \, \sqrt{\alpha}
					\, \int_1^{\infty}
									s^2 \, \exp\left(-\frac{s^2}{2}\right)
								\dint s \,.
	\end{equation*}
	
	Now, consider the Gaussian field~$\Psi$
	associated with the reproducing kernel~$K_d$.
	Using Dudley's result, here in the derived form~\eqref{eq:Dudley2},
	with \lemref{lem:E|X|<init+EsupX} we obtain
	\begin{align*}
		\expect \|\Psi\|_{\infty}
			&\leq \sqrt{\frac{2}{\pi}}
						+ 4 \, C_{\textup{Dudley}} \, C_3(p)
								\, \sqrt{d \, (1 + \log d)}
								\, \bigl(1 + \sqrt{\alpha} + \sqrt{- \log 2 R_0}\bigr) \\
			&\leq C_4(p) \, \sqrt{d \, (1 + \log d)}
									 \, \bigl(1 + \sqrt{\alpha} + \sqrt{- \log 2 R_0}\bigr) \,.
	\end{align*}
	By \lemref{lem:stdMCapp}, this gives us a final upper bound on the complexity.
	Concerning the $p$-dependence of the constant,
	we have $C(p) \asymp C_2(p) \preceq p^{-2} (1 + \log p^{-1})$.
\end{proof}

First we give sufficient conditions for the parameters~$\veclambda$
of the kernels~$K_{\veclambda}$
for that the conditions of \propref{prop:MCUBperiodic}
can be fulfilled with $p=1$.

\begin{lemma} \label{lem:kernelshape,r>1}
	Given a kernel
	\begin{equation*}
		K_1(x,y) := K_{\veclambda}(x,y)
			= \sum_{k=0}^{\infty} \lambda_k^2 \, \cos 2 \pi k (x-y)
	\end{equation*}
	with~\mbox{$\lambda_k \geq 0$},
	assume that the following holds:
	\begin{enumerate}[\quad(a)\quad]
		\item \label{enumcor:periodic,init}
			$\sum_{k=0}^{\infty} \lambda_k^2 = 1$,
		\item \label{enumcor:periodic,sum(k*la2)}
			$\sigma_{\veclambda}
				:= \sum_{k=1}^{\infty} k \, \lambda_k^2 < \infty$.
	\end{enumerate}
	Then the assumptions in \propref{prop:MCUBperiodic}
	are fulfilled with~$p=1$, $\alpha = 2 \pi \sigma_{\veclambda}$,
	and $R_0 = 1/2$.
	
	In particular, 
	for Korobov spaces~$\Hilbert_r^{\Korobov}(\Torus)$ with smoothness $r > 1$,
	we have \mbox{$\sigma_{\veclambda} = \beta_1 \, \zeta(2r-1) < \infty$}.
\end{lemma}
\begin{proof}
	Condition~\eqref{enumcor:periodic,init}
	is for the normalization of the initial error,
	see~\eqref{enum:MCUBperiodic,init} in~\propref{prop:MCUBperiodic}.
	
	Condition~\eqref{enumcor:periodic,sum(k*la2)} guarantees differentiability
	of~\mbox{$K_{\veclambda}(\cdot,0)$} with absolute convergence
	of the resulting series of sine functions, we have
	\begin{equation*}
		\frac{\diff}{\diff x} K(x,0)
			= - 2 \pi \sum_{k=1}^{\infty}
						k \, \lambda_k^2 \, \sin 2 \pi k x
			\geq - 2 \pi \sigma_{\veclambda} \,.
	\end{equation*}
	This directly implies \mbox{$K(x,0) \geq 1 - 2 \pi \sigma_{\veclambda} \, x$}
	for~\mbox{$x \geq 0$}.
\end{proof}

For Korobov spaces with lower smoothness
we need a more specific technique for estimating the shape of the kernel.

\begin{lemma} \label{lem:kernelshape,r<=1}
	Consider the kernel
	\begin{equation*}
		K_1(x,y) = \beta_0 + \beta_1 \sum_{k=1}^{\infty} k^{-2r} \cos 2\pi k (x-y)
	\end{equation*}
	of a one-dimensional Korobov space~$\Hilbert_r^{\Korobov}(\Torus)$
	with smoothness~\mbox{$1/2 < r \leq 1$},
	where \mbox{$\beta_0,\beta_1 > 0$}
	and \mbox{$\beta_0 + \beta_1 \, \zeta(2r) = 1$}.
	Then the assumptions in \propref{prop:MCUBperiodic}
	are fulfilled with~$p=2r-1$,
	a universal constant $\alpha > 0$,
	and $R_0 = 1/(\sqrt{2} \, \pi)$.
\end{lemma}
\begin{proof}
	Due to symmetry by translations,
	we only need to study the kernel function
	$K(x,0)$ for~$x \in [0,R_0]$, where $0 < R_0 \leq 1/2$
	will be determined later.
	By assumption we have~\mbox{$K(0,0) = \beta_0 + \beta_1 \, \zeta(2r) = 1$}.
	In order to understand the behaviour for~$x > 0$,
	we need to understand the oscillating series
	\begin{equation*}
		\sum_{k=1}^{\infty} k^{-2r} \cos 2\pi k x = \sum_{k=1}^{\infty} g_x(k) \,,
	\end{equation*}
	where we sum over discrete points of the function
	\begin{equation*}
		g_x(z) := z^{-2r} \cos 2\pi x z \,.
	\end{equation*}
	
	Observe that we have \mbox{$g_x(z) \geq 0$}
	for \mbox{$0 < z \leq 1/(4x)$}.
	From the series representation of the cosine function we further know
	\begin{equation*}
		g_x(z)
			\geq z^{-2r} \, \left(1 - \frac{(2\pi x z)^2}{2}\right)
			=: \tilde{g}_x(z) \,.
	\end{equation*}
	Here we have~\mbox{$\tilde{g}_x(z) \geq 0$}
	for \mbox{$0 < z \leq 1/(\sqrt{2} \, \pi x)$}.
	We assume $x \leq 1/(\sqrt{2} \, \pi)$
	so that at least $k=1$ lies within that range.
	Then
	\begin{align*}
		\sum_{0 < k \leq 1/(4x)} g_x(k)
			&\,\geq\, \sum_{0 < k \leq 1/(\sqrt{2} \, \pi x)}
									\tilde{g}_x(k) \\
			&\,=\, \zeta(2r)
								- \sum_{k > 1/(\sqrt{2} \, \pi x)} k^{-2r}
								- 2\pi^2 \, x^2
										\sum_{0 < k \leq 1/(\sqrt{2} \, \pi x)} k^{2-2r} \,.\\
	\intertext{%
	The summands of the series
	may be interpreted as midpoints of intervals
	of length~$1$, that way underestimating the integral of the convex function
	\mbox{$z \mapsto z^{-2r}$} on that interval.
	The finite sum within the third term
	can be replaced by an integral over the non-decreasing function
	\mbox{$z \mapsto z^{2-2r}$} since $1/2 < r \leq 1$.
	Hence we estimate}
			&\,\geq\, \zeta(2r)
									- \int_{1/(\sqrt{2} \, \pi x) - 1/2}^{\infty}
												z^{-2r}
											\dint z
									- 2 \pi^2 \, x^2
											\int_{1}^{1/(\sqrt{2} \,\pi x) + 1}
													z^{2-2r}
												\dint z \\
			&\,=\, \zeta(2r)
							- \frac{1}{2r-1} \left(\frac{1}{\sqrt{2} \, \pi x}
																			- \frac{1}{2}
																\right)^{1-2r}\\
			&\phantom{\,=\, \zeta(2r)}
							- \frac{2 \pi^2}{3-2r} \, x^2 \,
									\left(\left(\frac{1}{\sqrt{2} \,\pi x} + 1
												\right)^{3-2r} - 1
									\right) \,. \\
	\intertext{%
	With $x \leq 1/(\sqrt{2} \, \pi)$,
	and bearing in mind~$1-2r < 0 < 3-2r$, we can further simplify,}
			&\,\geq\, \zeta(2r)
								- \frac{1}{2r-1} \left(\frac{1}{2 \sqrt{2} \, \pi x}
																\right)^{1-2r}
								- \frac{2 \pi^2}{3-2r} \, x^2 \,
										\left(\frac{\sqrt{2}}{\pi x} \right)^{3-2r} \,. \\
	\intertext{%
	Recall the asymptotic behaviour of the Riemann zeta function beyond its simple pole,
	\mbox{$\zeta(2r) \geq (2r-1)^{-1}$}
	for~\mbox{$r > 1/2$},
	so we end up with}
		\sum_{0 < k \leq 1/(4x)} g_x(k)
			&\,\geq\, \zeta(2r) \, (1 - c_1 \, x^{2r-1}) \,,
	\end{align*}
	where~$c_1 > 0$ is a universal constant.
	
	We continue estimating the tail of the series,
	\begin{equation*}
		\sum_{k > 1/(4x)} g_x(k)
			\,\approx\, \int_{1/(4x)}^{\infty} g_x(z) \dint z \,.
	\end{equation*}
	Here, the integral is approximated by one function value~\mbox{$g_x(k_j)$}
	within each of the intervals
	\mbox{$I_j := \left(1/(4x) + j,\, 1/(4x) + j + 1\right]$},
	all of which have length~$1$,
	namely we have the value at
	\mbox{$k_j := \left\lfloor 1/(4x) + j + 1 \right\rfloor$},
	\mbox{$j \in \N_0$}.
	The local Lipschitz constant~$L_j$ of~\mbox{$g_x(z)$} on $I_j$
	is bounded by
	\begin{align*}
		L_j &\,\leq\, \left[ 2\pi x \, z^{-2r}
													+ 2r \, z^{-2r-1}
									\right]_{z=1/(4x) + j}\\
			&\,\leq\, 2 \, (\pi + 4r) \, x \, \left(j+1\right)^{-2r} \,,
	\end{align*}
	where we exploited $z \geq 1/(4x)$ and
	$0 < x \leq 1/(\sqrt{2} \, \pi) < 1/4$.
	This can be plugged into a basic result on integral approximation,
	\begin{equation*}
		\left\vert g_x(k_j) - \int_{I_j} g_x(z) \dint z
		\right\vert
			\,\leq\, \frac{L_j}{2} \,,
	\end{equation*}
	which is similar to the midpoint rule,
	and leads to the estimate
	\begin{align} \label{eq:tail_ser=int}
		\left| \sum_{k > 1/(4x)} g_x(k)
								- \int_{1/(4x)}^{\infty} g_x(z) \dint z
		\right|
			&\,\leq\, (\pi + 4r) \, x \sum_{j=0}^{\infty} \left(1 + j\right)^{-2r}
				\nonumber\\
			&\,=\, (\pi + 4r) \, \zeta(2r) \, x\,.
	\end{align}
	Since~\mbox{$\cos 2\pi x z \leq 0$}
	for~\mbox{$z \in \left[(4 l + 1)/(4x),\, (4 l + 3)/(4x)\right]$},
	\mbox{$l \in \N_0$}, we have
	\begin{equation} \label{eq:g(z)>}
		g_x(z) \,\geq\, \left(\frac{4 l + 1}{4x}\right)^{-2r} \, \cos 2\pi x z
	\end{equation}
	in that interval.
	Conversely, \mbox{$\cos 2\pi x z \geq 0$}
	for~\mbox{$z \in \left[(4 l - 1)/(4x),\, (4 l + 1)(4x)\right]$},
	\mbox{$l \in \N$},
	so \eqref{eq:g(z)>} holds accordingly.
	Together this gives
	\begin{align*}
		\int_{1/(4x)}^{\infty} g_x(z) \dint z
			&\,\geq\, \left(\frac{1}{4x}\right)^{-2r}
									\underbrace{\int_{1/(4x)}^{3/(4x)}
																	\cos 2\pi x z 
																\dint z
															}_{= - 1/(\pi x)}\\
			&\qquad
						\,+\, \sum_{l=1}^{\infty}
										\left(\frac{4 l + 1}{4x}\right)^{-2r}
											\underbrace{\int_{(4l-1)/(4x)}^{(4l+3)/(4x)}
																			\cos 2\pi x z 
																		\dint z
																	}_{ = 0}\\
			&\,=\, - \frac{4^{2r}}{\pi} \, x^{2r-1} \,.
	\end{align*}
	Combined with~\eqref{eq:tail_ser=int},
	and under the constraint~$0 < x \leq 1/(\sqrt{2} \, \pi)$,
	this yields
	\begin{align*}
		\sum_{k > 1/(4x)} g_x(k)
			&\,\geq\,
					- (\pi + 4r) \, \zeta(2r) \, x - \frac{4^{2r}}{\pi} \, x^{2r-1} \\
			&\,\geq\, - c_2 \, \zeta(2r) \, x^{2r-1}\,,
	\end{align*}
	where~$c_2 > 0$ is a universal constant.
	
	From these considerations, and recalling that
	\mbox{$\beta_0 + \beta_1 \, \zeta(2r) = 1$},
	for the kernel function we obtain
	\begin{align*}
		K(x,0)
			&\,=\, \beta_1 + \beta_2 \sum_{k=1}^{\infty} k^{-2r} \cos 2 \pi k x \\
			&\,\geq\, \beta_1 + \beta_2 \, \zeta(2r) (1 - \alpha \, x^{2r-1}) \\
			&\,\geq\, 1 - \alpha \, x^{2r-1} \,,
	\end{align*}
	where~$\alpha = c_1 + c_2 > 0$ is a universal constant,
	and \mbox{$0 \leq x \leq 1 / (\sqrt{2} \, \pi) =: R_0$}.
\end{proof}

\begin{remark}
	For $r=1$, a closed formula for the kernel is known, we have
	\begin{equation*}
		K(x,0) = \beta_1 + \beta_2 \, \pi^2 \, \left(\frac{1}{6} - x + x^2\right)
	\end{equation*}
	for~$x \in [0,1]$, see~\cite[25.12.8]{DLMF}.
	
	\lemref{lem:kernelshape,r>1} produces unpleasant constants
	if the smoothness~$r > 1$ is getting close to~$1$.
	This can be avoided by taking the approach from \lemref{lem:kernelshape,r<=1},
	just using
	\begin{equation*}
		\sum_{0 < k \leq 1/(\sqrt{2} \, \pi x)} k^{2-2r}
			\leq \frac{1}{\sqrt{2} \, \pi x}
	\end{equation*}
	instead of some integral approximation.
	Note that now no factor~$(3-2r)^{-1}$ will occur.
	In the end we will get the estimate
	\begin{equation*}
		K(x,0) \geq 1 - \alpha \, x \,,
		\qquad \text{for $0 \leq x \leq R_0$,}
	\end{equation*}
	with the same values for~$\alpha$ and~$R_0$
	as in \lemref{lem:kernelshape,r<=1}.
	
	Larger exponents~$p =2r - 1$ will be possible for~$1 < r < 3/2$.
	There we need to conduct a similarly sophisticated study of
	the oscillating series
	representing the derivative~\mbox{$\frac{\diff}{\diff x} K(x,0)$}.
	This, however, comes at the price of unpleasant constants, especially
	near the critical points~$1$ and $3/2$.
\end{remark}

Together with \propref{prop:MCUBperiodic},
Lemma~\ref{lem:kernelshape,r>1} and \ref{lem:kernelshape,r<=1}
directly lead to the following tractability result.

\begin{theorem} \label{thm:Korobov}
	Consider unweighted Korobov spaces~%
	\mbox{$\Hilbert_r^{\Korobov}(\Torus^d) = \Hilbert_{\veclambda}(\Torus^d)$}
	as described above.
	For smoothness~\mbox{$r > 1/2$},
	fixing~\mbox{$\beta_0,\beta_1 > 0$}
	such that the initial error is constant~$1$ for all dimensions,
	we have polynomial tractability for the uniform approximation
	with Monte Carlo methods that use linear information,
	in detail,
	\begin{equation*}
		n^{\ran}(\eps,\Hilbert_r^{\Korobov}(\Torus^d)
										\hookrightarrow L_{\infty}(\Torus^d)) \\
			\preceq \eps^{-2} \, d \, (1 + \log d) \,.
	\end{equation*}
	The hidden constant may depend on~$r$.
\end{theorem}

\begin{remark}[Loss of smoothness]
	\label{rem:SmoothnessLost}
	With the idealized algorithm from~\remref{rem:Lstoch}
	we loose smoothness~$1/2$.
	This points to the non-interpolatory nature of the approximation method we take.
	In detail, check that the Gaussian process~$\Psi$ associated
	to~$\Hilbert_r^{\Korobov}$ (for any equivalent norm)
	lies almost surely in~$\Hilbert_s^{\Korobov}$ for~\mbox{$r-s > 1/2$},
	and it is almost surely not in~$\Hilbert_s^{\Korobov}$
	for~\mbox{$r-s \leq 1/2$}.
	The argument is similar to that in \remref{rem:Lstoch},
	see also~\cite[Chapter~I,~{\S}2]{Kuo75}.
	The following picture is a simulation of a $1$-dimensional approximation on~\mbox{$\Hilbert_r^{\Korobov}(\Torus)$} with~\mbox{$r = 1.25$},
	\mbox{$\beta_0 = 0.4$}, and \mbox{$\beta_1 = 0.4473...$},
	and gives an idea of what the approximation looks like.
	Note that the method
	\mbox{$A_n^{\omega}(f)
					:= \frac{1}{n} \sum_{i=1}^n L_i^{\omega}(f) \, g_i^{\omega}$}
	depends on the chosen norm.
	
	\begin{center} \noindent
	\begin{tikzpicture}
	\begin{axis}[ xlabel = {$x$},
								ylabel = {$f(x)$, $[A_n(f)](x)$},
								xmin = 0, xmax = 1,
								ymin = -0.5, ymax = 1,
								x post scale = 1.90,
								y post scale = 1.1,
								legend pos=north west]
		\addplot [black,line width = 3pt] table {originalfcn.dat};
		\addplot [darkgrey,line width = 1pt] table {approx16.dat};
		\addplot [lightgrey,line width = 1.5pt] table {approx256.dat};
		\addlegendentry{input $f$}
		\addlegendentry{$A_n(f)$, $n=\phantom{0}16$}
		\addlegendentry{$A_n(f)$, $n=256$}
	\end{axis}	
	\end{tikzpicture}\\
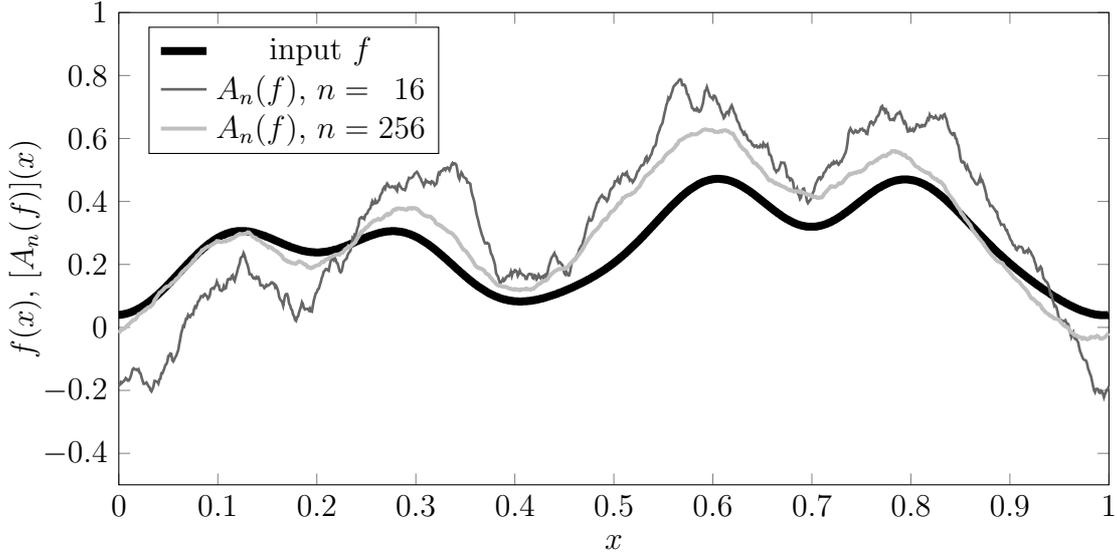
	\textbf{Figure~1.\;} Simulation of random approximation
	for $f \in \Hilbert_r^{\Korobov}(\Torus)$.
	\end{center}
\end{remark}

\begin{remark}[Combined methods]
	The upper bounds in \thmref{thm:Korobov} do not give the optimal order
	of convergence for the approximation of Korobov functions.
	The rate of convergence we can guarantee by this
	is only~\mbox{$e^{\ran}(n) \preceq n^{-1/2}$} for~\mbox{$r > 1/2$},
	where the implicit constant may depend on~$d$ and $r$.
	
	In fact, the optimal rate can be determined as
	\begin{equation*}
		(n^{-1} \, (\log n)^{d-1})^r
			\preceq
				e^{\ran}(n,\Hilbert_r^{\Korobov}(\Torus^d) \embed L_{\infty}(\Torus^d))
			\preceq (n^{-1} \, (\log n)^{d-1})^r \, \sqrt{\log n} \,,
	\end{equation*}
	with implicit constants depending on~$d$ and $r$,
	see the author's joint paper with Byrenheid and Nguyen~\cite{BKN17}.
	For similar results on $L_q$-approximation \mbox{$1<q<\infty$},
	cf.~Fang and Duan~\cite{FD07}.
	The algorithmic difference is that
	half of the information budget is spent on collecting exact knowledge
	about the most important Fourier coefficients,
	while the remaining Fourier coefficients are collectively covered
	by the fundamental Monte Carlo method presented in \secref{sec:plainMC}.
\end{remark}

\subsection{Final Remarks on the Initial Error}
\label{sec:HilbertFinalRemarks}

For unweighted tensor product problems as in this paper,
the assumption of a normalized initial error is crucial for the new
approach to work.
If not, that is, if
\begin{equation*}
	e(0,\Hilbert(K_1) \hookrightarrow L_{\infty}(D_1))
		= \sup_{x \in D_1} \sqrt{K_1(x,x)}
		=: 1 + \gamma > 1 \,,
\end{equation*}
then for the $d$-dimensional problem
we have an exponentially large initial error
\begin{equation*}
	e(0,\Hilbert(K_d) \hookrightarrow L_{\infty}(D_d))
		= (1 + \gamma)^d \,,
\end{equation*}
where~$K_d$ is the product kernel and~\mbox{$D_d := \bigtimes_{j=1}^d D_1$}.
For the Gaussian field~$\Psi$ with covariance function~$K_d$ this implies
\begin{equation*}
	\expect \|\Psi\|_{\infty}
		\geq \sup_{\vecx \in D_d} \expect |\Psi_{\vecx}|
		= \sqrt{\frac{2}{\pi}} \, (1 + \gamma)^d \,.
\end{equation*}
In this situation \lemref{lem:stdMCapp}
can only give an impractical upper complexity bound
which grows exponentially in~$d$ for fixed~\mbox{$\eps > 0$}.
However, if the constant function~\mbox{$f=1$}
is normalized in~\mbox{$\Hilbert(K_1)$},
but \mbox{$\Hilbert(K_1)$} is non-trivial
and contains more than just constant functions,
then~\mbox{$\sup_{x \in D_1} \sqrt{K_1(x,x)} > 1$}.
This would be contrary to our requirements.

Therefore we must accept that the constant function~\mbox{$f = 1$} cannot be
a normalized function in~$\Hilbert(K_d)$ if we want to
break the curse for the $L_{\infty}$-approximation
with the present tools.
This stands in contrast to several other often-studied problems:
\begin{itemize}
	\item
		Take the $L_2$~approximation of periodic Korobov spaces,
		see e.g.\ Novak and Wo\'zniakowski~\cite[pp.~191--193]{NW08}.
		The initial error is~$1$ iff the largest singular value is~$1$,
		so it is natural to let the constant function~\mbox{$f = 1$}
		be normalized within the input space~$\Hilbert$.
	\item
		Consider multivariate integration over Korobov spaces,
		see for example Novak and Wo\'zniakowski~\cite[Chapter~16]{NW10}.
		The integral is actually the Fourier coefficient
		belonging to the constant function.
		The initial error is properly normalized iff
		the constant function~\mbox{$f = 1$} has norm~$1$.
\end{itemize}

\section*{Acknowledgements}

This paper is based on a chapter of the author's PhD thesis~\cite{Ku17},
which has been financially supported by the DFG Research Training Group 1523.
I wish to express my deepest gratitude to my PhD supervisor Erich Novak
for many valuable hints during my work on these results.
I also wish to thank my colleagues David Krieg and Van Kien Nguyen
for making me aware of the situations
which have been discussed in \secref{sec:HilbertFinalRemarks}.

\renewcommand{\refname}{Bibliography}


\begin{thebibliography}{99.}
\bibliographystyle{plain}


\bibitem{Adl90}
	\textsc{R.J.~Adler}.
	\newblock An Introduction to Continuity, Extrema,
		and Related Topics for General Gaussian Processes.
	\newblock {\em IMS Lecture Notes-Monograph Series}, 12, 1990.


\bibitem{Ar50}
	\textsc{N.~Aronszajn}.
	\newblock Theory of reproducing kernels.
	\newblock {\em Transactions of the American Mathematical Society},
		68(3):337--404, 1950.

\bibitem{Bakh59}
	\textsc{N.S.~Bakhvalov}.
	\newblock On the approximate calculation of multiple integrals.
	\newblock {\em Vestnik MGU, Ser. Math. Mech. Astron. Phys. Chem.} 4:3--18,
		1959, in Russian.
	\newblock English translation: {\em Journal of Complexity}, 31(4):502--516, 2015.


\bibitem{BKN17}
	\textsc{G.~Byrenheid, R.J.~Kunsch, V.K.~Nguyen}.
	\newblock Monte Carlo Methods for the $L_{\infty}$-Approximation
		on Periodic Sobolev Spaces with Bounded Mixed Derivative.
	\newblock Submitted to {\em Journal of Complexity}, 2017.

\bibitem{CKS16}
	\textsc{F.~Cobos, T.~K\"uhn, W.~Sickel}.
	\newblock Optimal approximation of multivariate periodic
		Sobolev functions in the sup-norm.
	\newblock {\em Journal of Functional Analysis}, 270(11):4196--4212, 2016.

\bibitem{Dud73}
	\textsc{R.M.~Dudley}.
	\newblock Sample functions of the Gaussian process.
	\newblock {\em The Annals of Probability}, 1(1):66--103, 1973.


\bibitem{FD07}
	\textsc{G.~Fang, L.~Duan}.
	\newblock The information-based complexity of approximation problem
		by adaptive Monte Carlo methods.
	\newblock {\em Science in China Series A: Mathematics}, 51(9):1679--1689, 2008.


\bibitem{FR13CS}
	\textsc{S.~Foucart, H.~Rauhut}.
	\newblock {\em A Mathematical Introduction to Compressive Sensing}.
	\newblock Birkh\"auser, 2013.



\bibitem{He92}
	\textsc{S.~Heinrich}.
	\newblock Lower bounds for the complexity of Monte Carlo function
		approximation.
	\newblock {\em Journal of Complexity}, 8(3):277--300, 1992.

%
%

\bibitem{HeM11}
	\textsc{S.~Heinrich, B.~Milla}.
	\newblock The randomized complexity of indefinite integration.
	\newblock {\em Journal of Complexity}, 27(3--4):352--382, 2011.


\bibitem{HNUW17}
	\textsc{A.~Hinrichs, E.~Novak, M.~Ullrich, H.~Wo\'{z}niakowski}.
	\newblock Product rules are optimal for numerical integration
		in classical smoothness spaces.
	\newblock {\em Journal of Complexity}, 38:39--49, 2017.

\bibitem{HNW13}
	\textsc{A.~Hinrichs, E.~Novak, H.~Wo\'{z}niakowski}.
	\newblock Discontinuous information in the worst case and
		randomized settings.
	\newblock {\em Mathematische Nachrichten} 286:679-690, 2013.



\bibitem{Ka81}
	\textsc{B.~Kashin}.
	\newblock Widths of Sobolev classes of small-order smoothness.
	\newblock {\em Vestnik Moskovskogo Universiteta Seriya 1 Matematika Mekhanika},
		5:50--54, 1981, in Russian.


%

\bibitem{Ku16}
	\textsc{R.~J.~Kunsch}.
	\newblock Bernstein numbers and lower bounds for the Monte Carlo error.
	\newblock {\em Proceedings of the MCQMC 2014}, 471--488, 2016.

\bibitem{Ku17}
	{\sc R.J.~Kunsch}.
	\newblock {\em High-Dimensional Function Approximation:
		Breaking the Curse with Monte Carlo Methods}.
	\newblock PhD Thesis, FSU Jena, 2017, available on arXiv:1704.08213 [math.NA].

\bibitem{KWW08}
	\textsc{F.Y.~Kuo, G.W.~Wasilkowski, H.~Wo\'zniakowski}.
	\newblock Multivariate $L_{\infty}$ approximation in the worst case
		setting over reproducing kernel Hilbert spaces
	\newblock {\em Journal of Approximation Theory}, 152:135--160, 2008.

\bibitem{Kuo75}
	\textsc{H.H.~Kuo}.
	\newblock Gaussian Measures in Banach Spaces.
	\newblock {\em Lecture Notes in Mathematics}, 463, 1975.

\bibitem{LT91}
	\textsc{M.~Ledoux, M.~Talagrand}.
	\newblock \emph{Probability in Banach Spaces}.
	\newblock Ergebnisse der Mathematik und ihrer Grenzgebiete,
		3.~Folge, Band 23, Springer-Verlag, 1991.

\bibitem{Lif95}
	\textsc{M.A.~Lifshits}.
	\newblock {\em Gaussian Random Functions}.
	\newblock Mathematics and Its Applications, 
	Kluwer Academic Publishers, 1995.



\bibitem{Ma91}
	\textsc{P.~Math{\'e}}.
	\newblock Random approximation of Sobolev embeddings.
	\newblock {\em Journal of Complexity}, 7(3):261--281, 1991.

%

\bibitem{No88}
	\textsc{E.~Novak}.
	\newblock Deterministic and Stochastic Error Bounds in Numerical Analysis.
	\newblock {\em Lecture Notes in Mathematics}, 1349, 1988.

\bibitem{NW08}
	\textsc{E.~Novak, H.~Wo\'zniakowski}.
	\newblock {\em Tractability of Multivariate Problems},	Volume~I,
		Linear Information.
	\newblock European Mathematical Society, 2008. 

\bibitem{NW10}
	\textsc{E.~Novak, H.~Wo\'zniakowski}.
	\newblock {\em Tractability of Multivariate Problems},	Volume~II,
		Standard Information for Functionals.
	\newblock European Mathematical Society, 2010.

\bibitem{NW12}
	\textsc{E.~Novak, H.~Wo\'zniakowski}.
	\newblock {\em Tractability of Multivariate Problems},	Volume~III,
		Standard Information for Operators.
	\newblock European Mathematical Society, 2012.

\bibitem{DLMF}
	\textsc{F.W.J.~Olver, D.W.~Lozier, R.F.~Boisvert, C.W.~Clark}, editors.
	\newblock {\em NIST Handbook of Mathematical Functions}.
	\newblock Cambridge University Press, 2010.


\bibitem{OP95}
	\textsc{K.Y.~Osipenko, O.G.~Parfenov}.
	\newblock Ismagilov type theorems for linear, Gel'fand and Bernstein n-widths.
	\newblock {\em Journal of Complexity}, 11(4):474--492, 1995.

\bibitem{Pie74}
	\textsc{A.~Pietsch}.
	\newblock $s$-Numbers of operators in Banach spaces.
	\newblock {\em Studia Mathematica}, 51(3):201--223, 1974.
%
%
	

\bibitem{Rit00}
	\textsc{K.~Ritter}.
	\newblock Average-Case Analysis of Numerical Problems.
	\newblock {\em Lecture Notes in Mathematics}, 1733, 2000.


\bibitem{Smo65}
	\textsc{S.A.~Smolyak}.
	\newblock {\em On Optimal Restoration of Functions and Functionals of Them}.
	\newblock Candidate dissertation, Moscow State University, 1965, in Russian.
	
\bibitem{TWW88}
	\textsc{J.F.~Traub, G.W.~Wasilkowski, H.~Wo\'zniakowski}.
	\newblock {\em Information-Based Complexity}.
	\newblock Academic Press, 1988. 



\bibitem{Wah90}
	\textsc{G.~Wahba}.
	\newblock Spline Models for Observational Data.
	\newblock {\em CBMS-NSF Regional Conference Series in Applied Mathematics},
		1990.


\end{thebibliography}
\end{document}